\documentclass[11pt]{article}
\usepackage[utf8]{inputenc}
\usepackage[T1]{fontenc}
\usepackage{float}
\usepackage[final]{graphicx}
\usepackage{indentfirst}
\usepackage{fancyhdr}
\usepackage[a4paper]{geometry}
\usepackage{amsmath}
\usepackage{amssymb}

\usepackage{longtable}
\usepackage{booktabs}

\usepackage{listings}
\usepackage[unicode]{hyperref}
\usepackage{array}
\usepackage{pifont}
\usepackage{fancybox}
\usepackage{makecell}
\usepackage{xcolor}
\usepackage{siunitx}
\usepackage{arydshln}
\usepackage{eso-pic}
\usepackage{hyperref}
\usepackage{eurosym}
\usepackage{csquotes}
\usepackage{footmisc}
\usepackage{systeme}
\usepackage{amsthm}
\usepackage{tikz}
\usepackage{cleveref}
\usepackage{hyperref}
\usepackage{amsthm}
\usepackage{multirow}
\usetikzlibrary{intersections}
\usetikzlibrary{calc}
\usetikzlibrary{arrows.meta}     
\usetikzlibrary{decorations.pathreplacing}

\usepackage[most]{tcolorbox}
\usepackage{enumitem}

\usepackage{authblk}
\usepackage[authoryear,round]{natbib}

\usepackage{centernot}

\usepackage{caption}
\usepackage{subcaption}

\newcounter{example}

\providecommand{\keywords}[1]{%
  \par\smallskip\noindent\textbf{Keywords:} #1\par\smallskip}

\newcommand{\interp}[2]{\par\smallskip\noindent\textit{Interpretation#2 Example \ref{example_microgrid}: #1}\par\smallskip}

\newtheorem{theorem}{Theorem}

\tcolorboxenvironment{theorem}{
  colback=gray!10,
  colframe=gray!60,
  boxrule=0.8pt,
  arc=2mm,
  left=6pt,
  right=6pt,
  top=6pt,
  bottom=6pt
}

\tcolorboxenvironment{lemma}{
  colback=gray!10,
  colframe=gray!60,
  boxrule=0.8pt,
  arc=2mm,
  left=6pt,
  right=6pt,
  top=6pt,
  bottom=6pt
}

\tcolorboxenvironment{corollary}{
  colback=gray!10,
  colframe=gray!60,
  boxrule=0.8pt,
  arc=2mm,
  left=6pt,
  right=6pt,
  top=6pt,
  bottom=6pt
}

\tcolorboxenvironment{definition}{
  colback=gray!10,
  colframe=gray!60,
  boxrule=0.8pt,
  arc=2mm,
  left=6pt,
  right=6pt,
  top=6pt,
  bottom=6pt
}

\tcolorboxenvironment{remark}{
  colback=gray!10,
  colframe=gray!60,
  boxrule=0.8pt,
  arc=2mm,
  left=6pt,
  right=6pt,
  top=6pt,
  bottom=6pt
}

\tcolorboxenvironment{observation}{
  colback=gray!10,
  colframe=gray!60,
  boxrule=0.8pt,
  arc=2mm,
  left=6pt,
  right=6pt,
  top=6pt,
  bottom=6pt
}

\newtheorem{lemma}{Lemma}

\theoremstyle{definition}
\newtheorem{definition}{Definition}

\theoremstyle{remark}

\newtheorem{observation}{Observation}

\newcommand{\vct}[1]{\mathbf{#1}}

\newcommand{\y}{\vct{y}}
\renewcommand{\b}{\vct{b}}
\renewcommand{\c}{\vct{c}}
\renewcommand{\d}{\vct{d}}
\newcommand{\x}{\vct{x}}

\renewcommand{\a}{\vct{a}}

\renewcommand{\u}{\vct{u}}
\renewcommand{\v}{\vct{v}}
\renewcommand{\P}{\mathcal{P}}
\newcommand{\Q}{\mathcal{Q}}
\newcommand{\btheta}{\theta}
\newcommand{\bkappa}{\kappa}
\newcommand{\bnu}{\nu}

\title{Analyzing changes in optimal variables in linear programming with uncertain parameters}

\author[1,2]{Baptiste Istace}
\author[1]{Guillaume Derval}
\author[1]{Bardhyl Miftari}
\author[1]{Quentin Louveaux}
\affil[1]{University of Li\`ege, Montefiore Institute, Belgium}
\affil[2]{Corresponding author. \textit{Email address:} baptiste.istace@uliege.be}

\date{}

\begin{document}
\lhead{}
\rhead{}
\lfoot{}
\rfoot{}
\maketitle
\pagestyle{fancy}



\begin{abstract}
Linear problems often include many parameters that may be uncertain. Sensitivity analysis studies how these parameters impact optimal values. Instead of analyzing the objective function, we shift the focus to the optimal values of the variables. Three types of linear modifications are considered: on the cost vector, the right-hand side, and on the constraint matrix. Several theorems establish properties of these modifications, including conditions for continuity of optimal variable values, as well as local convexity and concavity properties.
\end{abstract}


\section{Introduction}
Optimization models are built on assumptions on key values such as costs, demand, revenues, deadlines, and available resources. The solution obtained is optimal only with respect to these assumptions. When they change significantly, the optimal solution obtained may become suboptimal or even infeasible. In other words, an optimal solution in a reference scenario can become suboptimal when the environment differs from the initial assumptions. 
In this paper, the set of optimal solutions for a linear programming problem is examined as a function of a parameter’s variation, with a focus on how the optimal decisions themselves change rather than on the evolution of the optimal objective value.

This kind of approach is important in applications where input parameters depend on external conditions, making optimization results more sensitive to the differences between the expected scenario and the observed scenario. For example, in the energy sector, the increasing use of renewable energies makes the operation of electrical systems more dependent on external variables such as temperature, wind, solar energy, and weather. These variables directly influence production levels, but also the balance between supply and demand, which affects price levels, adjustment requirements, and, more generally, operational decisions. \citet{MOSQUERALOPEZ2024107789} show that temperature, wind, solar radiation, and precipitation have a significant effect on electricity prices, with nonlinear relationships and particularly marked impacts during extreme events. In logistics and industry, perturbations on strategic sea routes such as the Red Sea and the Suez Canal have illustrated how geopolitical uncertainty can increase transit times, raise costs, and undermine "optimal" plans built on assumed stable deadlines and capacities. The Suez Canal blockage incident was studied in particular, quantifying its economic losses and the effects of delays on maritime transport networks \citep{TRAN2025109464}. 

More generally, optimization results depend on the values of the parameters used in the model. It therefore becomes necessary to examine how the solution responds to data perturbations and to identify the parameters that cause the most significant changes in the resulting decisions. This is precisely the approach taken in sensitivity analysis. This viewpoint is also related to robust and stochastic optimization, although the objective differs: rather than computing a single solution that performs well across a range of scenarios, it characterizes how the set of optimal solutions itself changes as a parameter varies. Robust optimization seeks solutions that remain feasible or perform well for all realizations within a given set of uncertain conditions, whereas stochastic optimization includes uncertainty through scenarios or probability distributions and optimizes an expected or stochastic criterion. In contrast, the goal is not to compute a robust stochastic solution. Rather, the study focuses on how the sets of optimal solutions for a deterministic linear program evolve when a parameter varies.

The field of sensitivity analysis studies the effect of variations in model parameters on the solution obtained. Traditionally, its goal is to observe how the optimal value obtained changes when the input data is slightly modified. This allows key parameters influencing the results to be identified, and system performance to be studied in different scenarios. However, in many applications, understanding the variation in the objective value is not sufficient. Indeed, decision makers need to know \emph{how the decisions themselves change}, and which variables (or combinations of variables) are responsible for these changes.
This decision-focused approach, which concentrates on the evolution of sets of optimal solutions rather than on optimal values alone, is the perspective adopted in this paper.

\subsection{Formalization and related works}
\noindent
Consider the standard linear optimization problem:
\begin{equation}
    \begin{aligned}
        \P \equiv \min_\x \quad & \c^t\x\\
        \text{s.t.} \quad & A\x\le \b \\
         & \x \geq 0 
    \end{aligned}
\end{equation}
where $\c \in \mathbb{R}^{n}$, $\x \in \mathbb{R}_+^{n}$, $\b \in \mathbb{R}^{m}$ and $A \in \mathbb{R}^{m\times n}$, with $m$ the number of constraints and $n$ the number of  variables.
Define the feasible region $K := \{\x\in\mathbb{R}^n_+ \mid A\x \le \b\}$. Such a model depends on external parameters, which are distributed into the three main parts of the model: the cost vector $\c$, the right-hand side $\b$, and the constraint matrix $A$. When some of these external parameters are unknown or subject to change, practitioners can be willing to study the impact of modifying $A$, $\b$, or $\c$, each corresponding to different sources of uncertainty:
\begin{itemize}
    \item Modifications of $\c$ allow to encode uncertainty of various costs (energy prices, transportation costs, installation costs, CAPEX, OPEX, penalties, ...);
    \item Modifications of $\b$ can be modifications of demands or known capacities;
    \item Modifications of $A$ modify the structure of the problem (by modifying the sparsity of $A$ and which variable are involved in which constraint), and can also model modification of yields, consumption coefficients, travel times, ...
\end{itemize}
This article studies how linear modifications of $A$, $\b$, and $\c$ affect the optimal solutions, with a particular focus on the optimal values assigned to each variable. Let $ \lambda \in \Lambda \subseteq \mathbb{R}$ be a scalar parameter that models uncertainty, and let $D \in \mathbb{R}^{m\times n}$ be the perturbation directions for the objective/right-hand side and for the constraint matrix, respectively. Note that by abuse of notation, we use $\d \in \mathbb{R}^{n}$ for cost perturbations and $\d \in \mathbb{R}^{m}$ for right-hand side perturbations.  Each possible modification is studied separately:

\begin{equation}
    \begin{minipage}[c]{0.32\linewidth}
        \centering
        \vspace{0pt}
        $\begin{aligned}
        \P_c(\lambda) \equiv & \min_\x  \  (\c+ \lambda \d)^t\x\\
        \text{s.t.} \quad & A\x\le \b \\
         & \x \geq 0 
        \end{aligned}$
    \end{minipage}
    \hfill
    \begin{minipage}[c]{0.32\linewidth}
        \vspace{0pt}
        \centering
        $\begin{aligned}
        \P_b(\lambda) \equiv & \min_\x \  \c^t\x\\
        \text{s.t.} \quad & A\x\le \b+  \lambda \d \\
         & \x \geq 0 
        \end{aligned}$
    \end{minipage}
    \hfill
    \begin{minipage}[c]{0.32\linewidth}
        \vspace{0pt}
        \centering
        $\begin{aligned}
        \P_A(\lambda) \equiv & \min_\x \  \c^t\x\\
        \text{s.t.} \quad & (A+\lambda D)\x\le \b \\
         & \x \geq 0 
        \end{aligned}$
    \end{minipage}
\end{equation}
For each of these perturbed problems, the associated optimal value functions are denoted by $f_c(\lambda)$, $f_b(\lambda)$, and $f_A(\lambda)$, and the corresponding feasible sets are denoted by $K_c(\lambda)$, $K_b(\lambda)$, and $K_A(\lambda)$. Under perturbations of the cost vector $\c$ ($\P_c(\lambda)$), the constraints remain unchanged, so that the set of feasible solutions does not vary with $\lambda$; consequently, $K_c(\lambda)=K,\ \forall \lambda$. On the other hand, perturbations of the RHS vector $\b$ or the constraint matrix $A$ affect the constraints themselves, so that the set of feasible solutions becomes dependent on $\lambda$. Accordingly, the feasible sets are defined as $K_b(\lambda) := \{\x\in\mathbb{R}^n \mid A\x \le \b+\lambda\d, \x \geq 0\}$ and $K_A(\lambda) := \{\x\in\mathbb{R}^n \mid (A+\lambda D)\x \le \b, \x \geq 0\}.$

The corresponding dual problems are also analyzed. The dual of $\P$ can be written as
\begin{equation}\label{eq:dualQ}
    \begin{aligned}
        \Q \equiv \max_{\y} \quad & \b^t\y\\
        \text{s.t.} \quad & A^t\y \le \c\\
         & \y \leq 0,
    \end{aligned}
\end{equation}
where $\y \in \mathbb{R}^{m}$ is the vector of dual variables and the dual feasible region is $K^D:= \{\y\in\mathbb{R}^m \mid A^t\y\le \c,\ \y\le 0\}$. Accordingly, the associated dual problem forms the following parametric problems: 
\begin{equation}
    \begin{minipage}[c]{0.32\linewidth}
        \centering
        \vspace{0pt}
        $\begin{aligned}
        \Q_c(\lambda) \equiv & \max_\y \ \b ^t\y\\
        \text{s.t.} \quad & A^t\y\le \c+ \lambda \d\\
         & \y \leq 0 
        \end{aligned}$
    \end{minipage}
    \hfill
    \begin{minipage}[c]{0.32\linewidth}
        \vspace{0pt}
        \centering
        $\begin{aligned}
        \Q_b(\lambda) \equiv & \max_\y \  (\b+  \lambda \d)^t\y\\
        \text{s.t.} \quad & A^t\y\leq \c  \\
         & \y \leq 0 
        \end{aligned}$
    \end{minipage}
    \hfill
    \begin{minipage}[c]{0.32\linewidth}
        \vspace{0pt}
        \centering
        $\begin{aligned}
        \Q_A(\lambda) \equiv & \max_\y \ \b^t\y\\
        \text{s.t.} \quad & (A+\lambda D)^t\y\le \c \\
         & \y \leq 0 
        \end{aligned}$
    \end{minipage}
\end{equation}

Under right-hand side perturbations ($\Q_b(\lambda)$), only the objective changes and the dual feasible set remains constant: $K_b^D(\lambda)=K^D$. On the other hand, cost vector $\c$ perturbation and matrix perturbation $A$ modify the dual constraints, yielding $K_c^D(\lambda):=\{\y\in\mathbb{R}^m \mid A^t\y\le \c+\lambda\d,\ \y\le 0\} \text{ and } K_A^D(\lambda):=\{\y\in\mathbb{R}^m \mid (A+\lambda D)^t\y\le \c,\ \y\le 0\}.$ In the literature, these problems have already been studied to some extent.\\

\noindent
\textit{Modification of either the objective $\c$ or the right-hand side $\b$.} Perturbations of the objective ($\P_c(\lambda)$) and of the right-hand side ($\P_b(\lambda)$) are closely connected through duality. Indeed, $\b$ appears in the dual objective ($\b^t \y$), while $\c$ appears in the dual constraints ($A^t \y \le \c$). Therefore, varying $\b$ in the primal corresponds to varying the objective of the dual problem ($\Q_b(\lambda)$), while varying $\c$ in the primal corresponds to shifting the right-hand side of the dual constraints ($\Q_c(\lambda)$).
In addition to this correspondence, both families benefit from well-known regularity properties. The optimal value function for both problems is piecewise linear on $\lambda$ (concave for $\P_c(\lambda)$ and convex for $\P_b(\lambda)$). These properties are formalized later in Theorems \ref{th:concave} and \ref{th:convexe} \citep[see, e.g.,][]{berkelaar1997optimal}.
Classical post-optimality analysis exploits the solver's results (optimal bases, reduced costs, dual multipliers, active sets) to compute the intervals of $\lambda$ over which a given basis remains optimal, and to restart re-optimization when a breakpoint is reached \citep[see][]{Gal1994Postoptimal,GeoffrionNauss1977Postoptimal}.\\

\noindent
\textit{Modification on the constraint matrix $A$.} Perturbations of the constraint coefficient matrix ($\P_A(\lambda)$) are more challenging. \citet{Zuidwijk2005LinearParametric} studies this case and derives local (basis-dependent) expressions for the optimal value as well as validity intervals. The method identifies an optimal basis, determines the range of $\lambda$ over which it remains optimal, computes the solution to the range, and performs a new optimization when a change of basis occurs. 
More precisely, fix $\lambda_0$ and let $\mathcal{B}$ be an optimal basis for $\P_A(\lambda_0)$. The same basis remains optimal if the basis submatrix $(A +\lambda D)_\mathcal{B}$ remains nonsingular, the associated basis solution remains feasible, and the reduced costs of the nonbasic variables remain nonnegative. These three conditions determine a validity interval $\Lambda_\mathcal{B}$ around $\lambda_0$, on which the optimal value can be written in the form :
\begin{equation}\label{eq:zuidwijk_rational}
    f_A(\lambda) = \frac{r(\lambda)}{q(\lambda)},
\end{equation}
where, once the basis $\mathcal{B}$ is fixed, $r(\lambda)$ and $q(\lambda)$ are known polynomials. Furthermore, $r(\lambda)$ and $q(\lambda)$ have degrees at most $m$ and $m+1$, respectively. When one of these conditions fails, the corresponding validity interval terminates, and a new optimal basis must be determined. It follows that the optimal value is a piecewise rational function of $\lambda$, each piece corresponding to an interval on which a given basis remains optimal. 

Related ideas have also been explored in process systems applications. \citet{KhalilpourKarimi2014LHS} propose algorithms for left-hand side uncertainty and rely on approximations of the basis inverse using the sensitivity update scheme of \citet{FlavellSalkin1975Approach}. More recently, \citet{MiftariEtAl2024LHSBounding} propose bounding schemes for the optimal solution of $\P_A(\lambda)$ under affine perturbations of the constraint matrix, combining robust optimization and Lagrangian relaxation ideas. This is done in order to construct constant, $\lambda$ dependent, and enveloping bounds. In a complementary direction, \citet{DervalMiftariErnstLouveaux2025_WarmstartA} develop efficient warm-starting strategies to evaluate multiple values of $\lambda$ by exploiting a precomputed basis and structured matrix updates.

Overall, these contributions mainly focus on the optimal objective value changes under perturbations.
The approach here is different.
The goal is to understand the behavior of $S_A(\lambda)$, $S_b(\lambda)$ and $S_c(\lambda)$, which correspond respectively to the set-valued mapping of optimal solutions to problems $\P_A(\lambda)$, $\P_b(\lambda)$ and $\P_c(\lambda)$. In other words, $S_c(\lambda):=\{\x\in K \mid (\c+\lambda \d)^t\x = f_c(\lambda)\},\ S_b(\lambda):= \{\x\in K_b(\lambda) \mid \c^t \x = f_b(\lambda)\} \text{ and } S_A(\lambda):= \{\x\in K_A(\lambda) \mid \c^t \x = f_A(\lambda)\}$.

Even for a fixed $\lambda_0$ of the parameter, describing the set of optimal solutions $S_A(\lambda_0)$, $S_b(\lambda_0)$ or $S_c(\lambda_0)$ is generally not trivial. Indeed, it amounts to identifying the optimal face of a polyhedron, and a complete characterization may require enumerating the vertices of this face, which is exponential in the worst case. Furthermore, an approach based on sampling $\lambda$ may reveal nonlinear behaviors (piecewise regimes, breaks, sudden changes in the optimal face) that would easily be missed or artificially smoothed out by a discrete sampling of $\lambda$ values.  In this paper, the objective is not to approximate $S_A(\lambda), S_b(\lambda)$ and $S_c(\lambda)$ through point exploration. 
The goal is to establish theorems that characterize the properties of these modifications, including continuity conditions, and local properties of convexity and concavity. In addition, representations (bands/envelopes or other structured descriptions) are proposed for each problem.


However, even when a complete description of $S(\lambda)$ is available, it is rarely directly usable by a decision maker. In practice, what matters is not the full geometry of the optimal set, but the value taken by specific decisions or combinations of decisions of interest. The analysis, therefore, focuses on a scalar criterion $\boldsymbol{\varphi}^t \x$, for a given vector $\boldsymbol{\varphi}$, which extracts from $S(\lambda)$ the quantity that is actually relevant to the decision maker. In each case, the evolution of a variable of interest over the optimal solutions $S(\lambda)$ is studied. Let $\boldsymbol{\varphi}^t\x$ denote this quantity of interest. This leads to the definition of
\begin{equation}\label{eq:problemgeneric}
    \begin{aligned}
        g(\lambda) = \min_{\x\in S(\lambda)} \quad & \boldsymbol{\varphi}^t\x.
    \end{aligned}
\end{equation}
$\boldsymbol{\varphi}^t\x$ could, for example, be equal to the value of one variable of interest $(\boldsymbol{\varphi}^t\x=x_i)$, or the mean of several variables ($\boldsymbol{\varphi}^t\x=\tfrac{1}{\left|  V_c\right|}\sum_{i\in V_c} x_i$ where $V_c$ is the set of selected variables), or the differences between two decisions $(\boldsymbol{\varphi}^t\x=x_i-x_j)$, or a weighted sum, a cost, impact, etc. 

Note that the same analysis is performed for the dual variables. The following model is also of interest: 
\begin{equation}\label{eq:problemgenericdual}
    \begin{aligned}
        h(\lambda) = \min_{\y\in T(\lambda)} \quad & \boldsymbol{\varphi}^t\y,
    \end{aligned}
\end{equation}
where $T(\lambda)$ denotes the set-valued mapping of optimal solutions to the dual parametric problem. This therefore corresponds to analyzing the behavior of $T_A(\lambda)$, $T_b(\lambda)$, or $T_c(\lambda)$, which correspond, respectively, to the set-valued mapping of optimal solutions to the dual problems $\Q_A(\lambda)$, $\Q_b(\lambda)$, or $\Q_c(\lambda)$.

\subsection{Our contributions}

This article develops a sensitivity analysis focused on decision-making for linear programs with linear perturbations of the cost vector $\c$, the right-hand side $\b$, and the constraint matrix $A$. We study how optimal decisions evolve with the parameter and how individual variables may vary between optimal solutions. The goal is to no longer analyze the objective, but rather the variables. 
More specifically, the work establishes theorems characterizing the structure of these optimal sets, with a focus on their continuity and local geometry.
We analyze not only sets-valued mapping of optimal solutions but also those of quasi-optimal solutions. When focusing on exact optimal solutions, the selected solutions may be discontinuous as the parameter changes, whereas allowing a small optimality tolerance restores continuity. Since an arbitrarily small tolerance is sufficient, and exact optimality is often obscured in practice by data uncertainty and solver precision, the difference between exact and quasi-optimal solutions is often negligible in practice. Decisions of interest are selected using a given additional criterion $\boldsymbol{\varphi}$. Our main contributions are as follows.

First, we formalize decision-oriented sensitivity analysis in linear programming for perturbations of $\c$, $\b$, and $A$ in both primal and dual form. This approach allows us to understand the evolution of the selected optimal or $\varepsilon$-optimal decisions when the parameter varies. Second, in the case of perturbations of the cost vector $\c$, we demonstrate that $g_c( \lambda)$ may be discontinuous. We prove that introducing a tolerance $\varepsilon>0$ guarantees continuity, and we derive, on the intervals where the basis remains unchanged, a homographic representation of the primal solution. Next, for the dual analysis, we show that $h_c(\lambda)$ remains continuous, including in the exact case $\varepsilon=0$. It admits a finite decomposition into intervals, on each of which it is a continuous convex piecewise-linear function. Then, for perturbations of the right-hand side $\b$, we show that the situation is reversed compared to cost perturbations. In this case, the asymmetry between primal and dual decision sensitivity appears in the opposite direction. Finally, for perturbations of the constraint matrix $A$, we establish that decision-oriented sensitivity becomes significantly more complicated. The matrix-perturbation case is the most challenging of the three perturbations. Table \ref{tab:placeholder} highlights the main properties of the functions in the different perturbation models.

\begin{table}[h]
    \centering
    \renewcommand{\arraystretch}{1.2}
    \begin{tabular}{|c|c|c|c|>{\centering\arraybackslash}m{7cm}|}
        \hline
        \multicolumn{2}{|c|}{Type of Sensitivity} & Function & Section & Main property \\
        \hline
        \multirow{2}{*}{$\c$} & Primal & $g_c(\lambda)$ & \ref{sub:primal_c} & Potential discontinuity; Continuity for $\varepsilon>0$; Piecewise homographic function. \\
        \cline{2-5}
        & Dual & $h_c(\lambda)$ & \ref{sub:dual_c} & Continuous, including for $\varepsilon=0$; Piecewise function by convex piecewise function. \\
        \hline
        \multirow{2}{*}{$\b$} & Primal & $g_b(\lambda)$ & \ref{sub:primal_b} & Potential discontinuity; Continuity for $\varepsilon>0$; Piecewise homographic function. \\
        \cline{2-5}
        & Dual & $h_b(\lambda)$ & \ref{sub:dual_b} &   Continuous, including for $\varepsilon=0$; Piecewise function by convex piecewise function. \\
        \hline
        $A$ & Primal/Dual & $g_A(\lambda), h_A(\lambda)$ & \ref{section:A} & Potential discontinuity, including $\varepsilon \geq 0$; Erratic function. \\
        \hline
    \end{tabular}
    \caption{Summary of the different methods discussed.}
    \label{tab:placeholder}
\end{table}

\subsection{Motivating example}

In this subsection, an example from the energy field is described to illustrate the concepts discussed in the following sections. As the paper progresses, this example is used to illustrate the differences between each type of problem. \\

\begin{example}\label{example_microgrid}
    Consider an electrical microgrid connected to the grid consisting of four elements: \begin{itemize}
        \item A solar photovoltaic (PV);
        \item A battery energy storage system;
        \item An inelastic electricity demand;
        \item A distribution network from which electricity can be imported.
    \end{itemize}
    The objective is to determine the optimal PV and battery capacity, as well as the hourly operating decisions of the microgrid, in order to minimize the total cost of the system over the study horizon.  At each stage, electricity demand must be satisfied. In this context, the primal variables represent design and operational decisions, such as installed photovoltaic capacity, installed battery capacity, and hourly grid imports. The dual variables, on the other hand, represent the marginal values associated with system constraints. For example, the dual variable associated with the power equation can be interpreted as the marginal value of electricity. In this context, the perturbations of $\c$, $\b$, and $A$ can be naturally interpreted as, for example, changes in economic conditions, demand, and technical coefficients, respectively.
\end{example}

\subsection{Paper outline}
The paper is organized as follows. In Section \ref{section:c}, the case where the coefficients of the objective function $\c$ are varied is studied. Then, Section \ref{section:b} looks at perturbations of the RHS $\b$ and Section \ref{section:A} focuses on perturbations of the matrix $A$. Finally, Section \ref{section:conclu} concludes by summarizing the main results and proposing perspectives, while Appendix \ref{appendix:A} provides a summary of the main notation.

\section{Sensitivity, when modifying $\c$}
\label{section:c}

In this section, the case where $\c$ is modified is analyzed. First, the parametric problem is reformulated into a more practical form. Next, the sensitivity properties for the primal problem $\P_c(\lambda)$ and finally for the dual problem $\Q_c(\lambda)$ are derived. This is done by describing how their respective sets of optimal solutions $S_c(\lambda)$ and $T_c(\lambda)$ evolve with the parameter $\lambda$, and by discussing the selection of representative optimal solutions. Let $f_c(\lambda)$ be the value function of the primal optimization problem $\P_c(\lambda)$:

\begin{equation}\label{eq:fc_primal}
    \begin{aligned}
        f_c(\lambda) = \min_{\x}\quad & (\c+\lambda \d)^{t}\x\\
        \text{s.t.}\quad & A\x \le \b\\
        & \x \ge 0,
    \end{aligned}
\end{equation}
where $A\in \mathbb{R}^{m\times n}$ is the constraint matrix, $\b\in \mathbb{R}^m$ the right-hand side, $\c \in \mathbb{R}^n $ the cost vector and $\d \in \mathbb{R}^n$ the vector affecting the objective.
The following theorem is well known \citep[see, for example,][]{berkelaar1997optimal,Bertsimas1997IntroductionTL}.
\begin{theorem}\label{th:concave}
Let $\Lambda\subseteq\mathbb{R}$ be an interval such that, for every $\lambda\in\Lambda$, the parametric problem $\P_c(\lambda)$ is feasible and has a finite optimal value. Then the value function $f_c(\lambda)$ is continuous, concave, and piecewise linear on $\Lambda$.
\end{theorem}

We will see in the next section that this is not true when focusing on variables rather than the value function.

\subsection{Sensitivity of primal variables}\label{sub:primal_c}
\interp{In this case, the primal variable is a design variable of the microgrid, such as the installed photovoltaic capacity or the installed battery capacity. The perturbation $\lambda \d$ in the cost vector models a change in economic conditions, for example, a variation in the grid electricity price or in the vector of investment costs.}{}

For a given vector of interest $\boldsymbol{\varphi}$, knowing that $f_c(\lambda)$ is the minimum value of $\P_c(\lambda)$,  \eqref{eq:problemgeneric} can  be rewritten as

\begin{equation}\label{eq:g_problem_c}
    \begin{aligned}
        g_c(\lambda) = \min_{\x} \quad & \boldsymbol{\varphi}^t\x\\
        \text{s.t.} \quad & (\c+\lambda \d)^{t} \x \leq f_c(\lambda) \\ & A \x\leq \b \\
         & \x \geq 0.
    \end{aligned}
\end{equation}

\subsubsection*{Continuity }\label{sub:continue}
\begin{observation}
    The value function $g_c(\lambda)$ can be discontinuous.
\end{observation}
\begin{example}\label{ex:1}
    Let us consider a simple 2D LP problem : 
    \begin{equation*}
        \begin{aligned}
            f_c(\lambda) = \min_{x_1,x_2}\quad & (2 + \lambda)\ x_1 + (1-2 \lambda)\  x_2 \\
            \text{s.t.}\quad
            & \left\{
            \begin{aligned}
                x_1 - 3x_2 &\le 0\\
                2x_1 - x_2 &\le 5\\
                -x_1 + 3x_2 &\le 3\\
                -2x_1 + x_2 &\le 0\\
                x_1,x_2 &\ge 0.
            \end{aligned}
            \right.
        \end{aligned}
    \end{equation*}    For $\boldsymbol{\varphi} = \begin{pmatrix}  1&0 \end{pmatrix}^t$, $g_c(\lambda)$ can be written as : 
    \begin{equation*}
        \begin{aligned}
            g_c(\lambda) &= \min_{x_1,x_2}\quad x_1 \\
            \text{s.t.}\quad
            & \left\{
            \begin{aligned}
                (2 + \lambda)\ x_1 + &(1-2 \lambda)\  x_2 \leq f_c(\lambda) \\
                x_1 - 3x_2 &\le 0\\
                2x_1 - x_2 &\le 5\\
                -x_1 + 3x_2 &\le 3\\
                -2x_1 + x_2 &\le 0\\
                x_1,x_2 &\ge 0.
            \end{aligned}
            \right.
        \end{aligned}
    \end{equation*}
\end{example}

The problem is represented in Figure \ref{fig:jump_gc}. Two values of $\lambda$ are considered. For $\lambda_1 = -7.1$, the optimal solution of $\P(\lambda)$ is $\x^*=(3\quad1)$ and the value of $g_c(-7.1)=3$. For $\lambda_2=-7$, the feasible region is represented in Figure \ref{figureb}. The optimal set $S_c(-7)$ is an edge. $\x^*=(3\ \ 1)$ is still an optimal solution but the optimal solution minimizing $x_1$ is now $\x^*=(0\ \ 0)$. $g_c(-7)$ is therefore 0.


\begin{figure}[h]
    \centering
    \begin{subfigure}[t]{0.32\linewidth}
        \centering
        \includegraphics[scale=0.345]{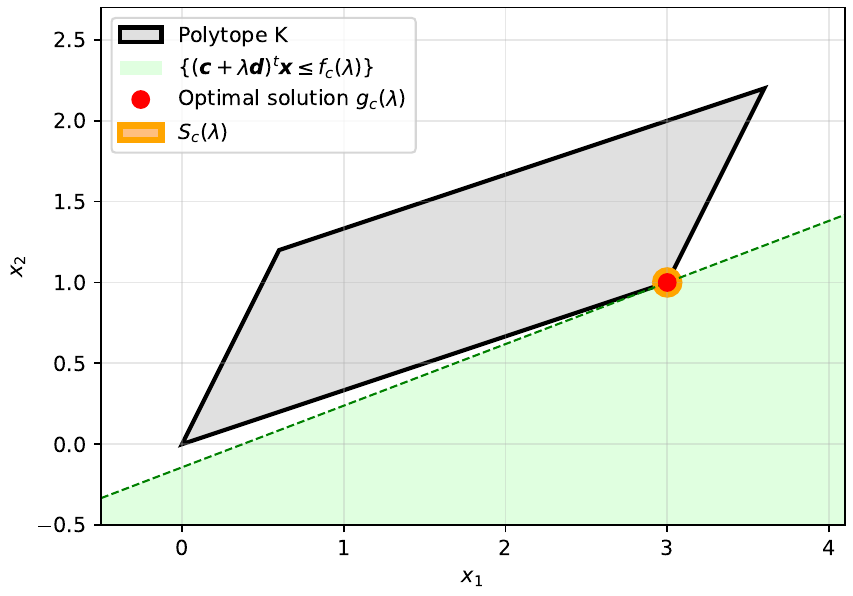}
        \caption{$\lambda_1=-7.1$ : the optimum is a single vertex, $g_c(-7.1)=3.$}
    \end{subfigure}\hfill
    \begin{subfigure}[t]{0.32\linewidth}
        \centering
        \includegraphics[scale=0.345]{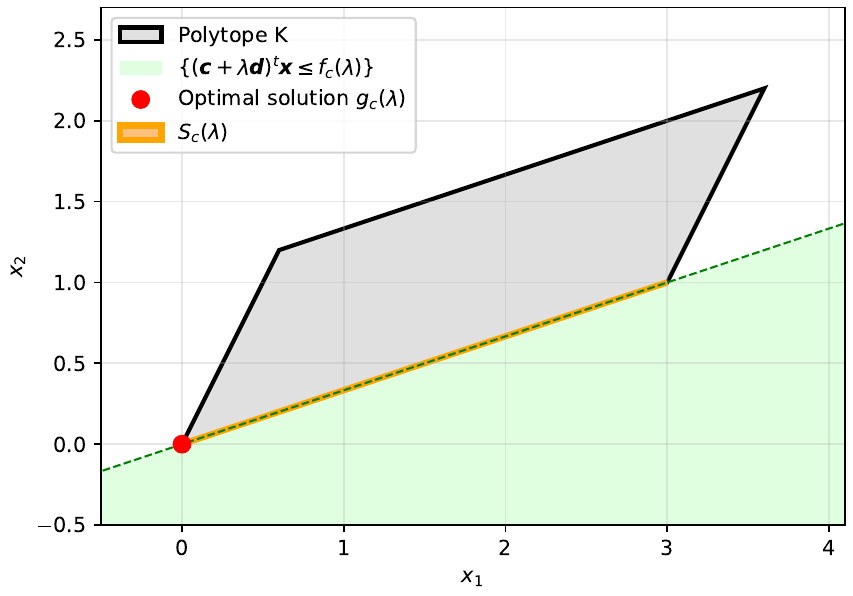}
        \caption{$\lambda_2 = -7$: the optimal set is an edge, $g_c(-7) = 0.$}
        \label{figureb}
    \end{subfigure}\hfill
    \begin{subfigure}[t]{0.32\linewidth}
        \centering
        \includegraphics[scale=0.32]{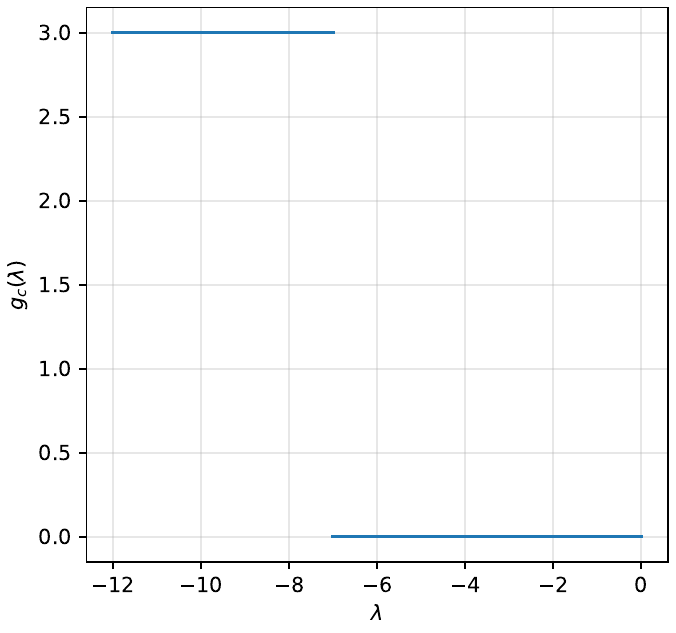}
        \caption{Representation of $g_c(\lambda)$.}
    \end{subfigure}
    \caption{Example of discontinuity for Example \ref{ex:1}.}
    \label{fig:jump_gc}
\end{figure}

This phenomenon can be observed from a geometrical perspective. As $\lambda$ varies, the hyperplane $\{\x \geq 0 \mid (\c+\lambda \d)^{t} \x \leq f_c(\lambda) \}$ rotates around the polytope $K$, shifting the face on which the optimum is attained. In general, this optimal face can be of arbitrary dimension $k\geq0$. However, nothing guarantees that this dimension remains constant in $\lambda$. Indeed, at a critical value $\lambda_2$, the optimal face may abruptly change dimension. This is precisely what occurs in the example above. For $\lambda_1 = -7.1$, the optimum is attained at a single vertex, whereas at $\lambda_2 = -7$, the optimal set $S_c(-7)$ becomes an entire edge. This transition from a vertex to an edge causes the minimizer of $x_1$ to jump abruptly from $(3\ \ 1)$ to $(0\ \ 0)$, producing the observed discontinuity in $g_c(\lambda)$.

However, if a tolerance on the optimality is considered, continuity is recovered. Let $\varepsilon>0$, and let \eqref{eq:g_problem_c} slightly modify into:
\begin{equation}\label{eq:g_problem_c_epsilon}
    \begin{aligned}
        g_c^\varepsilon(\lambda) = \min_{\x} \quad & \boldsymbol{\varphi}^t\x\\
        \text{s.t.}\quad & (\c+\lambda \d)^{t} \x \leq f_c(\lambda) + \varepsilon \\ & A \x \leq \b \\
         & \x \geq 0.
    \end{aligned}
\end{equation}
Figure \ref{fig:discontinuity} shows the effect of introducing $\varepsilon > 0$ on the continuity of $g_c^\varepsilon(\lambda)$. When $\varepsilon = 0$, the set of solutions can exhibit abrupt changes as $\lambda$ varies, leading to discontinuities. However, when $\varepsilon > 0$, these abrupt changes are smoothed out, resulting in a continuous function.

\begin{figure}[H]
    \centering    \includegraphics[width=0.6\textwidth]{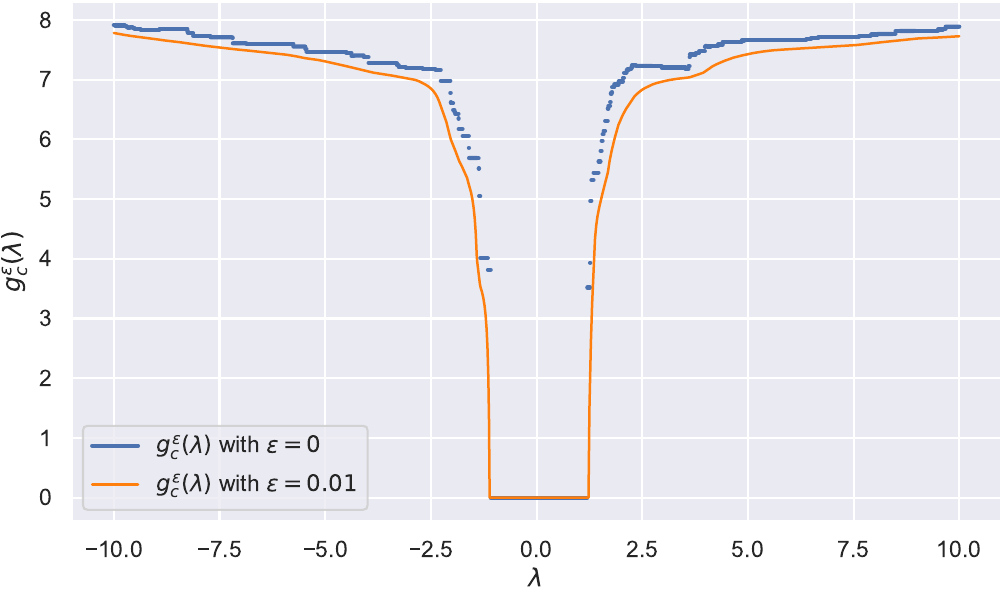}
    \caption{Representation of the function $g^\varepsilon_c(\lambda)$ for $\varepsilon = 0$ (in blue) and $\varepsilon > 0$ (in orange). The introduction of $\varepsilon$ smooths out discontinuities for the function as $\lambda$ varies. Example from a model with $1000$ variables and $600$ randomly selected constraints.}
    \label{fig:discontinuity}
\end{figure}

The main result of this subsection is the following theorem.

\begin{theorem}\label{Berge_maximum}
    If $\varepsilon>0$, the value function $g^\varepsilon_c(\lambda)$ is continuous at every $\lambda \in \Lambda$ .
\end{theorem}

The proof is given later in this subsection (on page \pageref{proof:berge}), after recalling the required notions on set-valued mappings and Berge's maximum theorem needed for the proof of Theorem \ref{Berge_maximum}.

According to Theorem \ref{Berge_maximum}, continuity is established once a solution is accepted whose objective value is arbitrarily close to the exact optimum $f_c(\lambda)$. The result is true for any strictly positive $\varepsilon$, however small it may be.  From an application perspective, the introduction of this tolerance is entirely natural.  Indeed, data (measurements, forecasts, parameters) and the solution provided by the numerical solver are affected by inevitable uncertainties and errors (measurement noise, rounding, optimality tolerances). Consequently, small optimal gaps are generally indistinguishable from experimental noise and numerical error. The requirement for "exact" optimality, therefore, often has no operational significance. 

In order to prove Theorem \ref{Berge_maximum}, some concepts from set-valued mapping are needed. See \citet{Berge1963} for more details. Let $X$ and $Y$ be topological spaces and let $\Gamma: X \rightrightarrows Y$ be a set-valued mapping, i.e., for each $\x\in X$, $\Gamma(\x)\subseteq Y$. Let $x_0\in X$.

\begin{definition}\label{def:USC}
    $\Gamma$ is \textit{upper semi-continuous (USC) at} $x_0$ if for each open set $V$ containing $\Gamma (x_0)$ there exists a neighborhood $U(x_0)$ such that $$x \in U(x_0)\quad \Rightarrow \quad \Gamma (x) \subset V.$$
\end{definition}

\begin{definition}\label{def:LSC}
    $\Gamma$ is \textit{lower semi-continuous (LSC) at} $x_0$ if for each open set $V$ that intersects $\Gamma (x_0)$ there is a neighborhood $U(x_0)$ such that $$x \in U(x_0) \quad \Rightarrow \quad \Gamma (x) \cap V \neq \varnothing.$$
\end{definition}

\begin{definition}
    The mapping $\Gamma$ is \textit{continuous at} $x_0$ if it is both lower and upper semi-continuous at $x_0$.
\end{definition}

\begin{theorem}[Berge's maximum theorem \citealp{Berge1963}]\label{th:bergemaximum}
    If $\psi:Y\rightarrow\mathbb{R}$ is a continuous numerical function in $Y$ and $\Gamma$ is a continuous mapping of $X$ into $Y$ such that, for each $\x$, $\Gamma (\x)\neq\varnothing$ :
    \begin{enumerate}[label=(\roman*)]
        \item The numerical function $M$ defined by $M(\x)=\max\{\psi(\y)\mid \y\in \Gamma (\x)\}$ is continuous in $X$.
        \item The mapping $\Psi$ defined by $ \Psi(x)=\{\,\y\mid \y\in \Gamma (\x),\ \psi(\y)=M(\x)\,\} $ is an upper semicontinuous mapping of $X$ into $Y$.
    \end{enumerate}
\end{theorem}

Once Theorem \ref{th:bergemaximum} and these definitions are in place, to prove Theorem \ref{Berge_maximum}, it must be shown that the set-valued mapping of optimal solutions $S_c^\varepsilon(\lambda)$ is upper semi-continuous and lower semi-continuous, which is done in Lemmas \ref{T:USC} and \ref{T:LSC}. This will allow us to establish the continuity of $S_c^\varepsilon(\lambda)$ (see Theorem \ref{T:P}), which follows from the proof of Theorem \ref{Berge_maximum}.

\begin{lemma}\label{T:USC}
    For every $\varepsilon\ge 0$, the set-valued mapping
    $S_c^\varepsilon(\lambda)=\{\x\in K\mid (\c+\lambda \d)^t x \le f_c(\lambda)+\varepsilon\}$ is upper semi-continuous at every $\lambda \in \Lambda$.
\end{lemma}
\begin{proof}
    Apply Theorem \ref{th:bergemaximum} with $X=\Lambda$, $Y=K\times\Lambda$ and define $\Gamma(\lambda) = K \times \{\lambda\}$. Since $K$ is a nonempty polytope, hence $\Gamma(\lambda)$ is nonempty for every $\lambda \in \Lambda$, and $\Gamma$ is continuous. Adding the fact that $\psi(\x,\lambda) = (\c+\lambda\d)^t\x$ is a continuous function in $Y$, the second part of Berge's maximum theorem can be applied. 
\end{proof}

\begin{observation}   
    It can also be seen that when $\varepsilon = 0$, the set-valued mapping $S^0_c(\lambda)$ is upper semi-continuous.
\end{observation}

Figure \ref{fig:uppersemicontinuity} highlights that, for $\lambda'$,  small perturbations of $\lambda$, the set $S^0_c(\lambda')$ remains contained in the open set $V$, even when the optimal face changes.


\begin{lemma}\label{T:LSC}
    The set-valued mapping $S_c^\varepsilon(\lambda)=\{\x\in K \mid (\c+\lambda \d)^t \x \le f_c(\lambda)+\varepsilon\}$ is lower semi-continuous at every $\lambda \in \Lambda$ if $\varepsilon > 0$. 
\end{lemma}

\begin{proof}
    To prove that $S_c^\varepsilon(\lambda)$ is lower semi-continuous at $\lambda$, it must be shown that it satisfies the following. Let V be an open set with $S^\varepsilon_c(\lambda)\cap V \neq \varnothing$, then there exists a neighborhood $U$ of $\lambda$ such that $S^\varepsilon_c(\lambda')\cap V \neq \varnothing$, for all $ \lambda'\in U$.\\
    Let us define $m(\x,\lambda)\equiv(\c+\lambda \d)^{t}\x$ and let $x_{0}\in S^\varepsilon_c(\lambda)\ \cap\ V$, so that 
    $$m(x_{0},\lambda) \le f_c(\lambda)+\varepsilon.$$
    Two cases arise : 
    \[
    \text{(i) } m(x_{0},\lambda) < f_c(\lambda) + \varepsilon, 
    \qquad
    \text{(ii) } m(x_{0},\lambda) = f_c(\lambda) + \varepsilon.
    \]
    
    In case (i), let $l(\lambda)=m(x_{0},\lambda)-f_c(\lambda)-\varepsilon < 0$. By continuity of $m(x_{0},\lambda)$ and $f_c(\lambda)$, $l$ is continuous, so there exists a neighborhood $U$ of $\lambda$ such that $l(\lambda')<0$ for all $\lambda'\in U$, i.e., $m(x_{0},\lambda')<f_c(\lambda')+\varepsilon$. Thus, $x_{0}\in S_c(\lambda')\cap V$ for all $\lambda'\in U$.
    
    In case (ii), let $x^{*} \in S^\varepsilon_c(\lambda)$ and  $ x^{*} \neq x_0 $ satisfy $m(x^{*},\lambda)=f_c(\lambda)$. For $t\in(0,1)$, introduce
    $y(t)=t x_{0}+(1-t)x^{*}$. By linearity,
    \[
    m(y(t),\lambda)=(\c+\lambda \d)^{t}(t x_{0} + (1-t)x^{*})= t\,m(x_{0},\lambda)+(1-t)\,m(x^{*},\lambda)
    =f_c(\lambda)+t\varepsilon<f_c(\lambda)+\varepsilon.
    \]
    
    Since $x_{0}\in V$ and $V$ is open, there exists $t$ close enough to $1$ such that $y(t)\in V$. Set $x_{1}=y(t)$, then $x_{1}\in V$ and $m(x_{1},\lambda)<f_c(\lambda)+\varepsilon$, the case is therefore reduced to (i).
\end{proof}

\begin{figure}[htpb]
    \centering
    \begin{subfigure}{0.45\textwidth}
        \centering
        \includegraphics[width=1\linewidth]{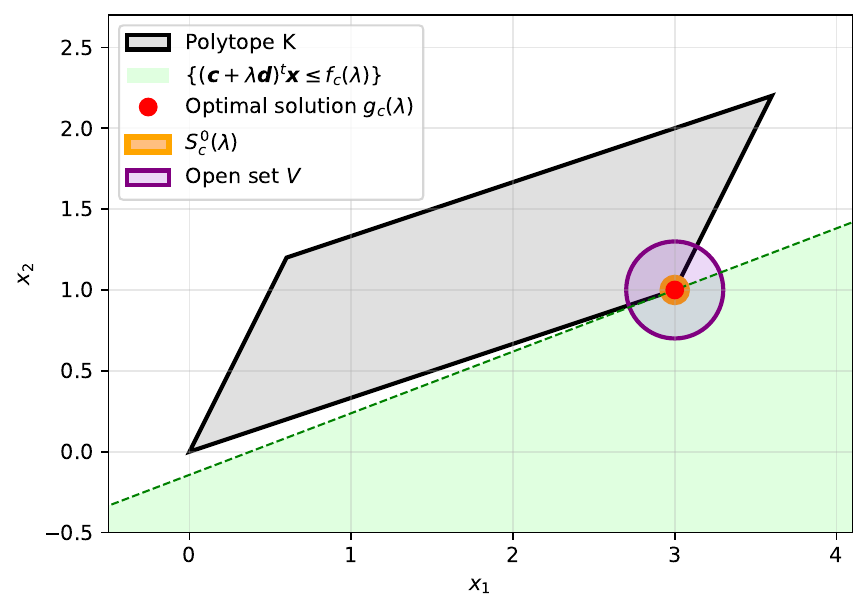}
        \caption{For $\lambda_1$, smaller than the critical value $\lambda_0$: the optimal solution is at the same vertex of $K$ and stays inside the open set $V$.}
        \label{fig:uhc1}
    \end{subfigure}
    \hfill
    \begin{subfigure}{0.45\textwidth}
        \centering
        \includegraphics[width=1\linewidth]{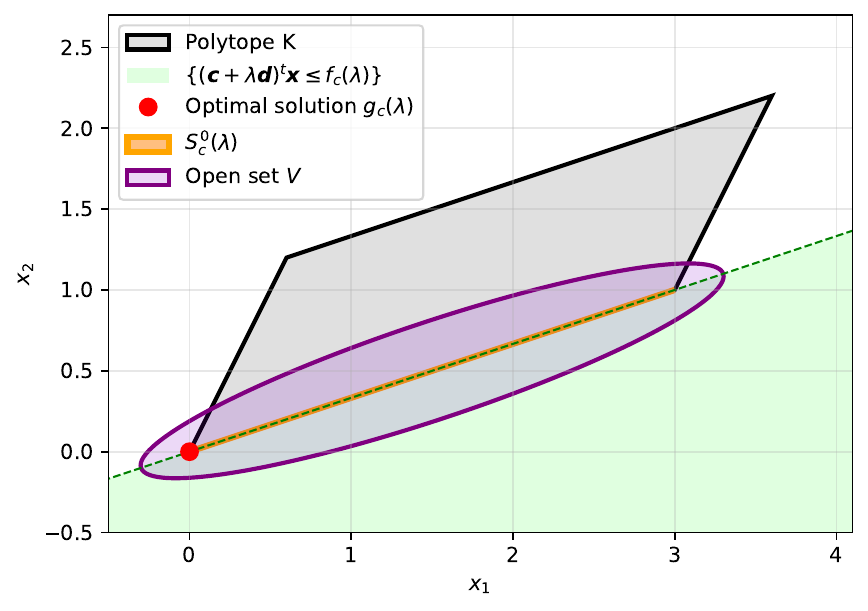}
        \caption{At the critical value of $\lambda_0$: the set $(\c + \lambda_0 \d)^t \x \leq f_c(\lambda)$ intersects a larger face of $K$, corresponding to a transition between optimal faces.}
        \label{fig:uhc2}
    \end{subfigure}
    \vskip\baselineskip
    \begin{subfigure}{\textwidth} 
        \centering
        \begin{minipage}{0.45\textwidth}
            \centering
            \includegraphics[width=\linewidth]{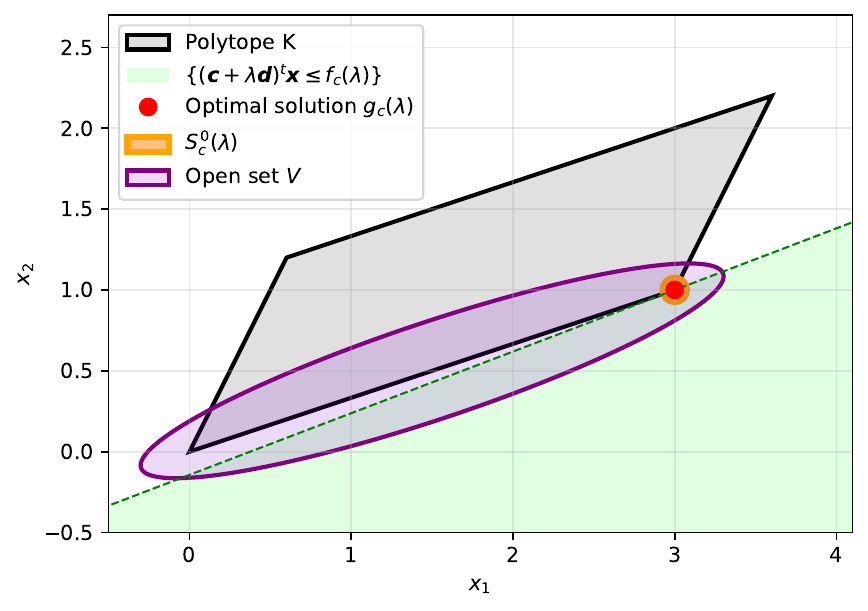}
        \end{minipage}
        \hfill
        \begin{minipage}{0.45\textwidth}
            \centering
            \includegraphics[width=\linewidth]{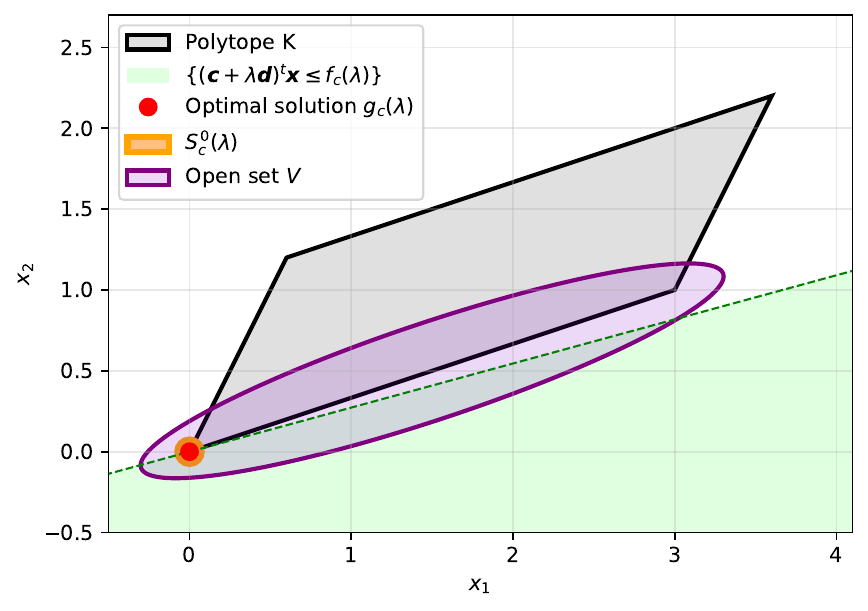}
        \end{minipage}
        \caption{
            For $\lambda'$, smaller than $\lambda_0$ 
            (see left subfigure) or bigger (see right subfigure),
            the sets $S^0_c(\lambda')$ remain entirely contained in the open set $V$.
        }
        \label{fig:uhc3_4}
    \end{subfigure}
    \caption{Illustration of upper semi-continuity when $\varepsilon = 0$. Despite the change in the optimal face of the polytope $K$, the set $S^0_c(\lambda)$ remains locally contained in the open set $V$. Indeed, for any $\lambda$ slightly larger than $\lambda_1$ in Figure \ref{fig:uhc1}, the situation in Figure \ref{fig:uhc2} never occurs locally. For any perturbation of $\lambda_1$, however small, it is always possible to choose a $\delta>0$ small enough to prevent $S^0_c(\lambda')$ from reaching the configuration shown in Figure \ref{fig:uhc2}. In other words, infinitesimal changes in $\lambda$ keep $S^0_c(\lambda')$ inside the open set $V$. At the critical case $\lambda_0$, the geometry of the level set $(\c+\lambda \d)^t\x\le f_c(\lambda)$ ensures that $S^0_c(\lambda)$ remains in $V$, which contains the entire region in which the optimal face changes. Indeed, for every value of $\lambda$ smaller or larger than $\lambda_0$, $V$ encompasses all $S^0_c(\lambda)$ (see Figures \ref{fig:uhc2} and \ref{fig:uhc3_4}).  These observations confirm that, even without tolerance $\varepsilon$,  the correspondence is indeed upper semi-continuous.}
    \label{fig:uppersemicontinuity}
\end{figure}
Figure \ref{fig:lowersemicontinuity_epsilon} illustrates a visual example with $\varepsilon>0$ in which $S_c^\varepsilon(\lambda)$ exhibits the behavior required by Definition \ref{def:LSC}, and is therefore lower semi-continuous in this setting. 

\begin{figure}[htbp]
    \centering
    \begin{subfigure}{0.45\textwidth}
        \centering
        \includegraphics[width=1\linewidth]{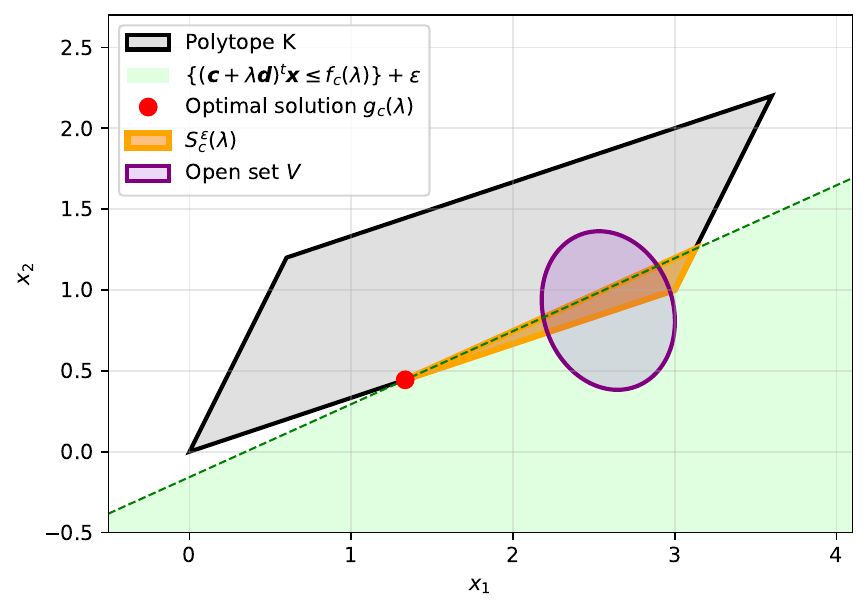}
    \end{subfigure}
    \hfill
    \begin{subfigure}{0.45\textwidth}
        \centering
        \includegraphics[width=1\linewidth]{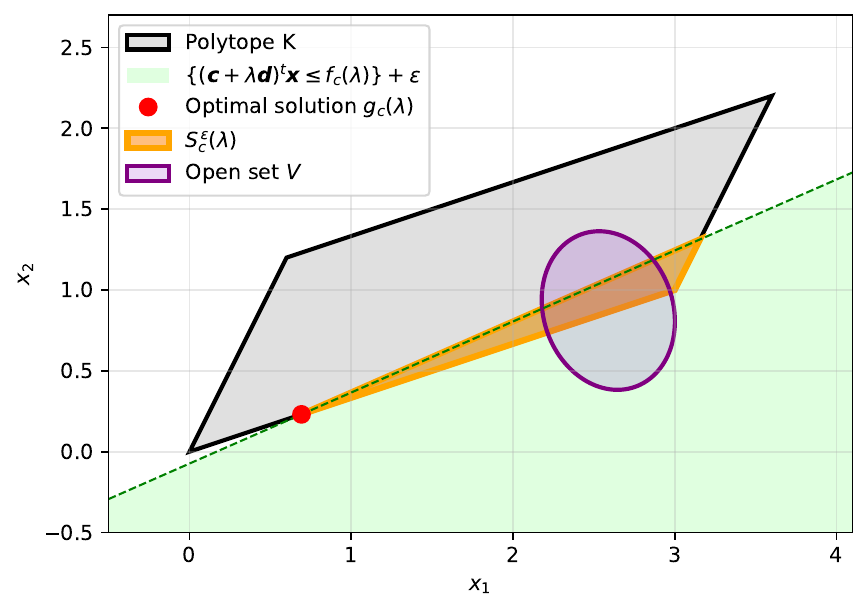}

    \end{subfigure}
    \caption{Illustration of lower semi-continuity for $\varepsilon > 0$. (Left) At a given value of $\lambda$, the optimal set $S^\varepsilon_c(\lambda)$ intersects the open set V. (Right) For a neighboring value $\lambda'$, the tolerance $\varepsilon$ ensures that $S^\varepsilon_c(\lambda') \cap V = S^\varepsilon_c(\lambda) \cap V \neq \varnothing$ remains within the open set $V$. This contrasts with the case $\varepsilon = 0$ (see Figure \ref{fig:lowersemicontinuity}.}
    \label{fig:lowersemicontinuity_epsilon}
\end{figure}

However, this is true only when $\varepsilon>0$; Figure \ref{fig:lowersemicontinuity} illustrates a situation in which the definition of lower semi-continuity is not satisfied.

\begin{observation}
    When $\varepsilon = 0$, the set-valued mapping $S^0_c(\lambda)$ is not lower semi-continuous at every $\lambda\in\Lambda$.
\end{observation}

\begin{figure}[htpb]
    \centering
    \begin{subfigure}{0.45\textwidth}
        \centering
        \includegraphics[width=1\linewidth]{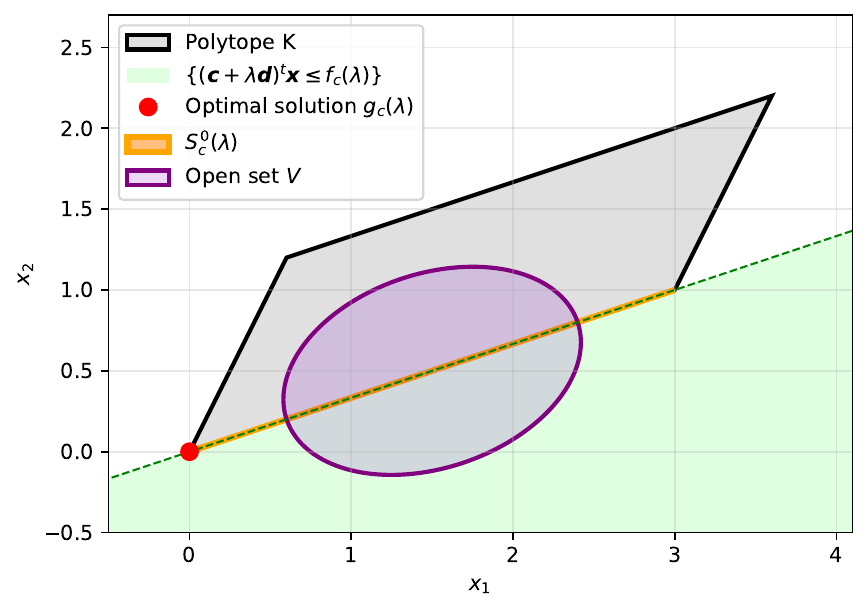}
    \end{subfigure}
    \hfill
    \begin{subfigure}{0.45\textwidth}
        \centering
        \includegraphics[width=1\linewidth]{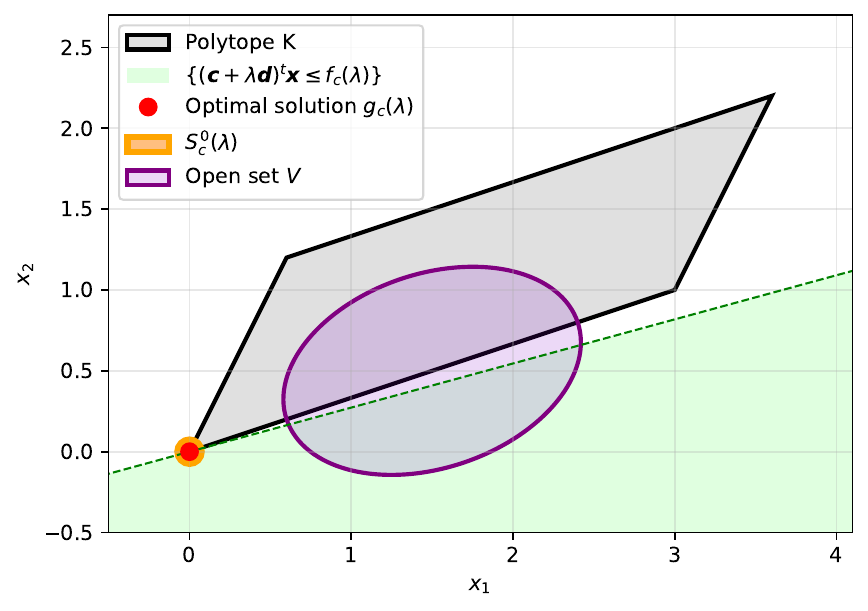}
    \end{subfigure}
    \caption{Example illustrating lower non-semi-continuity when $\varepsilon = 0$. In the figure on the left, the critical value $\lambda$ and the set of optimal solutions $S_c^0(\lambda)$ intersect the open set $V$.
    However, in the figure on the right, for a neighboring value $\lambda'$, the optimal face collapses into a vertex outside $V$, so that $S_c^0(\lambda') \cap V = \varnothing$. This shows the absence of lower semi-continuity.}
    \label{fig:lowersemicontinuity}
\end{figure}

\begin{theorem}\label{T:P}
    The set-valued mapping $S^\varepsilon_c(\lambda)$ is continuous at every $\lambda \in \Lambda$ if $\varepsilon > 0$.
\end{theorem}
\begin{proof}
    By definition, $S^\varepsilon_c(\lambda)$ is continuous if it is upper semi-continuous and lower semi-continuous. Lemmas \ref{T:USC} and \ref{T:LSC} prove respectively that it is upper and lower semi-continuous for $\varepsilon > 0$.
\end{proof}
It is now possible to return to Theorem \ref{Berge_maximum}, previously described, and prove it as a direct consequence of the continuity of optimal solution mapping $S^\varepsilon_c(\lambda)$.

\begin{proof}[Proof of Theorem \ref{Berge_maximum}]\label{proof:berge}
    According to Theorem \ref{th:bergemaximum}, the continuity of $g_c^\varepsilon(\lambda)$ is guaranteed provided that the set-valued mapping
    $S^\varepsilon_c : \Lambda \rightrightarrows X$ is nonempty-valued and \emph{continuous}, while the value function $\psi : X \to \mathbb{R}$, defined by $\psi(\x) = \boldsymbol{\varphi}^t \x$ is continuous. 
    Since
    \[
    g_c^\varepsilon(\lambda)
    = \min_{\x\in S^\varepsilon_c(\lambda)} \boldsymbol{\varphi}^t\x
    = -\,\max_{\x\in S^\varepsilon_c(\lambda)} \bigl(-\boldsymbol{\varphi}^t\x\bigr),
    \]
    Berge's maximum theorem applies directly to $-\boldsymbol{\varphi}^t\x$ and yields the continuity of
    $g_c^\varepsilon(\lambda)$, provided that $S^\varepsilon_c(\lambda)$ is continuous (which is proven by Theorem \ref{T:P}).
\end{proof}

Now that the continuity of $g_c^\varepsilon(\lambda)$ has been established for each $\varepsilon>0$, there is no practical reason to avoid adding $\varepsilon$. As explained above, it is reasonable to assume an arbitrarily low tolerance. Indeed, data are uncertain, solver results are affected by finite precision, and near-optimal solutions are often indistinguishable from exact solutions in practice.
\subsubsection*{Equivalent reformulation}\label{sub:reformulation}

In this subsection, the representation of $g_c^\varepsilon(\lambda)$ is analyzed in more detail. We also show that $g_c^\varepsilon(\lambda)$ can be reformulated as the optimal value of a parametric linear program of type $\P_A(\lambda) $, in which the parameter $\lambda$ appears in the constraint matrix $A$. In other words, even if $\lambda$ initially perturbs only the cost vector $\c$, the decision-oriented reformulation transfers this perturbation to the constraint matrix, so that $g_c^\varepsilon(\lambda)$ is equivalent to a problem of type $\P_A(\lambda)$.

Using Theorem \ref{th:concave}, $f_c(\lambda)$ is a piecewise linear concave function and can therefore be written as the minimum of multiple linear functions, i.e. $f_c(\lambda) = \min \{ \alpha_i \lambda + \beta_i \mid i \in 1...k \},\  \forall \lambda$. Moreover, since $t\leq \min u_i \Longleftrightarrow t \leq u_i$ for all $i$, the constraint $(\c+\lambda\d)^t\x \leq f_c(\lambda)+\varepsilon$ is equivalent to the finite number of linear functions $(\c+\lambda \d)^t\x \le \alpha_i \lambda + \beta_i+\varepsilon, \quad \forall i \in 1...k$.

\begin{observation}\label{th:reformulationstep1}
     \eqref{eq:g_problem_c_epsilon} is equivalent to the following problem :
    \begin{equation}
        \begin{aligned}\label{eq:g_problem_c_final}
            g_c^\varepsilon(\lambda) = \min_\x& \quad \boldsymbol{\varphi}^t\x\\
            \text{s.t. } &(\c+\lambda \d)^t\x \le \alpha_i \lambda + \beta_i+\varepsilon, \quad \forall i \in 1...k  \\ & A\x \leq \b \\
             & \x \geq 0
        \end{aligned}
    \end{equation}
    where $f_c(\lambda) = \min \{ \alpha_i \lambda + \beta_i \mid i \in 1...k \},\  \forall \lambda \in \Lambda$.
\end{observation}

Observation \ref{th:reformulationstep1} shows that \eqref{eq:g_problem_c_epsilon} can be rewritten by adding a finite number of additional constraints. Based on this reformulation, two points of view can then be considered. On the one hand, a local approach,  based on a decomposition of $\Lambda$ into intervals $I_j$ on which the linear piece of $f_c(\lambda)$ and the optimal basis remain unchanged, leads to the homographic expression given by Theorem \ref{th:g_cpiece}. On the other hand, a global approach, which allows the problem to be interpreted as a parametric problem of type $\P_A(\lambda)$, is established in Observation \ref{obs:equivalentA}.

\begin{theorem}\label{th:g_cpiece}
    For an interval $I_j$, where $\lambda \in [\lambda_j,\lambda_{j+1}]$ over which the same basis $\mathcal B$ remains optimal, the optimal value function $g_c^{\varepsilon(j)}(\lambda)$ is a linear-fractional function of $\lambda$ and is expressed in terms of three parameters as $$g_c^{\varepsilon(j)}(\lambda)=\btheta_{\mathcal B}+\frac{\lambda \bkappa_{\mathcal B}}{1+\lambda \bnu_{\mathcal B}},$$
    where $\btheta_{\mathcal B}, \bkappa_{\mathcal B}, \bnu_{\mathcal B}\in\mathbb{R}$ depend on the basis $\mathcal B$, and
    $$ \btheta_{\mathcal B}:=\boldsymbol{\varphi}^tA_{\mathcal B,0}'^{-1}\b_0,
    \qquad \bkappa_{\mathcal B}:=\boldsymbol{\varphi}^t(\alpha_j-\v^tA_{\mathcal B,0}'^{-1}\b_0)\,A_{\mathcal B,0}'^{-1}\u,
    \qquad \bnu_{\mathcal B}:=\v^tA_{\mathcal B,0}'^{-1}\u.$$ In these definitions, $A'_{\mathcal B,0}:=\begin{pmatrix}A_{\mathcal B}\\ \c_{\mathcal B}^t\end{pmatrix}$ denotes the augmented basis matrix, while $\u:=\begin{pmatrix}0\\1\end{pmatrix}$ and $\v^t:=\begin{pmatrix}0 & \d_{\mathcal B}^t \end{pmatrix}$ specify the rank-one perturbation. Moreover $\b_0:=\begin{pmatrix}\b\\ \beta_j+\varepsilon\end{pmatrix}$.
\end{theorem} 
\begin{proof}
    For an interval $I_j$ on which a given basis $\mathcal B$ remains optimal remains optimal, $f_c^{(j)}(\lambda)=\alpha_j\lambda+\beta_j$, since $x_{\mathcal B}=A_{\mathcal B}^{-1}b$ is independent of $\lambda$, giving $f_c^{(j)}(\lambda)=\c^t \x_{\mathcal B}+\lambda\, \d^t \x_{\mathcal B}$ (see Theorem \ref{th:concave}).
    The value function $g_c^{\varepsilon(j)}(\lambda)$ is thus equivalent to this representation :
    \begin{equation*}
        \begin{aligned}
            g_c^{\varepsilon(j)}(\lambda) = \min_\x& \quad \boldsymbol{\varphi}^t\x\\
            \text{s.t. }   & A\x \leq \b \\ 
            &(\c+\lambda \d)^t\x \le \alpha_j \lambda + \beta_j + \varepsilon \\
             & \x \geq 0.
        \end{aligned}
    \end{equation*}
    The solution $x_c^{\varepsilon(j)}(\lambda)$ for a given basis $\mathcal B$ is given by $x_c^{\varepsilon(j)}(\lambda)=A_\mathcal B'(\lambda)^{-1} \b'(\lambda)$, where $A_\mathcal B'(\lambda)$ is the augmented matrix associated with basis $\mathcal B$ and $\b'(\lambda)$ is the augmented right-hand side vector, both depending on $\lambda$.
    {\footnotesize
    \begin{equation*}
        \begin{aligned}
            x_c^{\varepsilon(j)}(\lambda)=
        \begin{bmatrix}
            \begin{pmatrix}
                A_\mathcal B\\\c_\mathcal B^t
            \end{pmatrix}
            + \lambda
            \begin{pmatrix}
                0\\\d_\mathcal B^t
            \end{pmatrix}
        \end{bmatrix}^{-1}
        \begin{bmatrix}
            \begin{pmatrix}
                \b\\\beta_j + \varepsilon
            \end{pmatrix}
            + \lambda
            \begin{pmatrix}
                0\\\alpha_j
            \end{pmatrix}
        \end{bmatrix} &=
        \begin{bmatrix}
            \underbrace{\begin{pmatrix}
                A_\mathcal B\\\c_\mathcal B^t
            \end{pmatrix}}_{A'_{\mathcal B,0}}
            + \lambda \underbrace{\begin{pmatrix}
                0\\1
            \end{pmatrix}}_{\u}
            \underbrace{\begin{pmatrix}
                0&\d_\mathcal B^t
            \end{pmatrix}}_{\v^t}
        \end{bmatrix}^{-1}
        \begin{bmatrix}
            \underbrace{\begin{pmatrix}
                \b\\\beta_j + \varepsilon
            \end{pmatrix}}_{\b_0}
            + \lambda \alpha_j
            \underbrace{\begin{pmatrix}
                0\\1
            \end{pmatrix}}_{\u}
        \end{bmatrix}\\
        & := \begin{bmatrix}
            A'_{\mathcal B,0} + \lambda\u\v^t
        \end{bmatrix}^{-1}\begin{bmatrix}
            \b_0+\lambda\alpha_j\u
        \end{bmatrix}
        \end{aligned}
    \end{equation*}}
    Using \citet{Sherman_Morrison_1949} and 
    if $1+\lambda\,\v^tA_{\mathcal B,0}'^{-1}\u\neq 0$, then
    $$
    (A'_{\mathcal B,0}+\lambda \u \v^t)^{-1}
    =
    A_{\mathcal B,0}'^{-1}
    -
    \frac{\lambda\,A_{\mathcal B,0}'^{-1}\u \v^t A_{\mathcal B,0}'^{-1}}{1+\lambda\, \v^tA_{\mathcal B,0}'^{-1}\u}.
    $$
    Note that $1+\lambda\,\v^tA_{\mathcal B,0}'^{-1}\u\neq 0$ holds automatically on $I_j$: by the matrix determinant lemma, $\det(A'_{\mathcal B,0}+\lambda \u\v^t)=\det(A'_{\mathcal B,0})(1+\lambda\,\v^tA_{\mathcal B,0}'^{-1}\u)$, so $1+\lambda\,\v^tA_{\mathcal B,0}'^{-1}\u=0$ would imply that $A'_{\mathcal B}(\lambda)$ is singular, contradicting the optimality (hence validity) of basis $\mathcal B$ on $I_j$.
    \noindent
    Hence,
    \begin{equation*}
        \begin{aligned}
            x_c^{\varepsilon(j)}(\lambda) & = \left(A_{\mathcal B,0}'^{-1} - \frac{\lambda\,A_{\mathcal B,0}'^{-1}\u \v^t A_{\mathcal B,0}'^{-1}}{1+\lambda\, \v^tA_{\mathcal B,0}'^{-1}\u} \right) (\b_0+\lambda\alpha_j\u)\\
            &= A_{\mathcal B,0}'^{-1}\b_0 + \lambda\alpha_j\,A_{\mathcal B,0}'^{-1}\u - \frac{\lambda\,A_{\mathcal B,0}'^{-1}\u \v^t A_{\mathcal B,0}'^{-1}\b_0} {1+\lambda\, \v^tA_{\mathcal B,0}'^{-1}\u} - \frac{\lambda^2\alpha_j\,A_{\mathcal B,0}'^{-1}\u \v^t A_{\mathcal B,0}'^{-1}\u} {1+\lambda\, \v^tA_{\mathcal B,0}'^{-1}\u}\\
            &= A_{\mathcal B,0}'^{-1}\b_0 + \lambda\alpha_j\,A_{\mathcal B,0}'^{-1}\u - \frac{\lambda\,A_{\mathcal B,0}'^{-1}\u} {1+\lambda\, \v^tA_{\mathcal B,0}'^{-1}\u}\left(\v^t A_{\mathcal B,0}'^{-1}\b_0+\lambda\alpha_j\,\v^t A_{\mathcal B,0}'^{-1}\u\right)\\
            &= A_{\mathcal B,0}'^{-1}\b_0 + \frac{\lambda\,A_{\mathcal B,0}'^{-1}\u}{1+\lambda\, \v^tA_{\mathcal B,0}'^{-1}\u}\left[\alpha_j\left(1+\lambda\, \v^tA_{\mathcal B,0}'^{-1}\u\right)-\left(\v^t A_{\mathcal B,0}'^{-1}\b_0+\lambda\alpha_j\,\v^t A_{\mathcal B,0}'^{-1}\u\right)\right]\\
            &= A_{\mathcal B,0}'^{-1}\b_0 + \frac{\lambda\left(\alpha_j-\v^t A_{\mathcal B,0}'^{-1}\b_0\right)} {1+\lambda\, \v^tA_{\mathcal B,0}'^{-1}\u}\,A_{\mathcal B,0}'^{-1}\u.
        \end{aligned}
    \end{equation*}
    Since $\boldsymbol{\varphi}$ does not depend on $\lambda$, it follows that
    $$
    g_c^{\varepsilon(j)}(\lambda)
    =
    \boldsymbol{\varphi}^t x_c^{\varepsilon(j)}(\lambda)
    =
    \boldsymbol{\varphi}^tA_{\mathcal B,0}'^{-1}\b_0
    + \boldsymbol{\varphi}^t\frac{\lambda\left(\alpha_j-\v^t A_{\mathcal B,0}'^{-1}\b_0\right)}
    {1+\lambda\, \v^tA_{\mathcal B,0}'^{-1}\u}
    \,A_{\mathcal B,0}'^{-1}\u.
    $$
    Therefore, defining
    $$
    \btheta_{\mathcal B}:=\boldsymbol{\varphi}^tA_{\mathcal B,0}'^{-1}\b_0,
    \qquad
    \bkappa_{\mathcal B}:=\boldsymbol{\varphi}^t
    \left(\alpha_j-\v^tA_{\mathcal B,0}'^{-1}\b_0\right)
    A_{\mathcal B,0}'^{-1}\u,
    \qquad
    \bnu_{\mathcal B}:=\v^tA_{\mathcal B,0}'^{-1}\u,
    $$
    it follows that
    $$
    g_c^{\varepsilon(j)}(\lambda)=\btheta_{\mathcal B}+\frac{\lambda \bkappa_{\mathcal B}}{1+\lambda \bnu_{\mathcal B}}.
    $$
\end{proof}

For an interval $I_j$, Theorem \ref{th:g_cpiece} indicates that the value function $g_c^{\varepsilon(j)}(\lambda)$ takes on a homographic form. More precisely, as long as the same basis $\mathcal B$ remains optimal, $g_c^{\varepsilon(j)}(\lambda)$ is determined by only three parameters, namely $\btheta_{\mathcal B}$, $\bkappa_{\mathcal B}$, and $\bnu_{\mathcal B}$. Thus, for each basis, the expression of the value function over the interval where this basis is optimal is directly obtained, which makes its computation very efficient. Indeed, starting from an optimal basis at a given $\lambda_0$, the solution can be computed until optimality is lost, which determines the next breaking point and the next basis. At each basis change, only the three parameters $(\btheta_{\mathcal B},\bkappa_{\mathcal B},\bnu_{\mathcal B})$ need to be updated to obtain the new expression. Since $\Lambda=\bigcup_{j=1}^p I_j$, where $p$ denotes the number of breakpoints of $f_c(\lambda)$, and since on each interval $I_j$ the function $g_c^{\varepsilon(j)}(\lambda)$ has a homographic representation as long as the same basis $\mathcal B$ remains optimal, it follows that $g_c^{\varepsilon}(\lambda)$ is a continuous piecewise homographic function on $\Lambda$.
However, the number of these intervals, and thus the number of pieces in this function, can grow exponentially in the worst case. 

\begin{theorem}[\citealp{Murty1980}]\label{th:murty}
    There exists a family of single parameter parametric linear programs $\P_c(\lambda)$ for which $\Lambda$ is partitioned into $2^n$ piecewise linear segments. Consequently, $\P_c(\lambda)$ may have an exponential number of breakpoints (equivalently, piecewise linear segments or basis changes) in the worst case. The same worst case phenomenon also holds for problem $\P_b(\lambda)$.
\end{theorem}

Consequently, although Theorem \ref{th:g_cpiece} provides a local description of $g_c^{\varepsilon(j)}(\lambda)$ on each interval where the optimal basis is fixed, a complete piecewise analysis may require solving an exponential number of subproblems in the worst case. This observation motivates a global approach. Rather than following the solution basis by basis, one can reformulate the problem over the entire parameter domain, as established in Observation \ref{obs:equivalentA}.

\begin{observation}\label{obs:equivalentA}
    \eqref{eq:g_problem_c_final} can be equivalently rewritten as a linear program in which the parameter $\lambda$ appears only on the left-hand side of the constraints $(\P_A(\lambda))$. This is followed by introducing an auxiliary variable $w$ such that $w=1$, leading to the formulation
    \begin{equation}
        \begin{aligned}
            g_c^\varepsilon(\lambda)=\min_{\x,w}\;& \boldsymbol{\varphi}^t \x\\
            \text{s.t.}\;& (A' + \lambda D')
            \begin{pmatrix}
                \x\\
                w
            \end{pmatrix}
            \le \b'\\
            & \x,w \ge 0,
        \end{aligned}
    \end{equation}
    with $$A' =
    \begin{pmatrix}
    A & 0\\
    \c & 0\\
    \c & 0 \\
    \vdots & \vdots \\
    0 &1\\
    0 &-1
    \end{pmatrix},
    \quad
    D' =
    \begin{pmatrix}
    0 & 0\\
    \d & -\alpha_1 \\
    \d & -\alpha_2 \\
    \vdots & \vdots \\
    0 & 0\\
    0 & 0
    \end{pmatrix},
    \quad
    \b' =
    \begin{pmatrix}
    \b \\
    \beta_1 + \varepsilon \\
    \beta_2 + \varepsilon\\
    \vdots \\
    1 \\
    -1
    \end{pmatrix}.$$
\end{observation}


Observation \ref{obs:equivalentA} shows precisely that this selection problem $g_c^\varepsilon(\lambda)$ can be reformulated as the optimal value of a parametric problem of type $\P_A(\lambda)$ (i.e., where the parameter $\lambda$ appears in the constraint matrix, in an augmented space). In other words, even if only $\c$ is modified at the outset, the transition to a decision-oriented analysis (via $g^\varepsilon_c(\lambda)$) reveals a perturbation on the left-hand side.

This observation provides a natural hierarchy of difficulty. When focusing only on the value function associated with $\P_c(\lambda)$, it comes back to the classical framework of objective perturbations, with a piecewise linear concave value function and breakpoints at basis changes. On the other hand, for $g^\varepsilon_c(\lambda)$, the reformulation of type $\P_A(\lambda)$ implies potentially more complex behavior, similar to that studied for matrix perturbations. Locally (fixed basis), the expressions become rational, and the domains of validity are governed by feasibility and optimality conditions \citep[see Zuidwijk's local expression and validity intervals, Eq. \eqref{eq:zuidwijk_rational}, ][]{Zuidwijk2005LinearParametric}. Nevertheless, in this case, the local description is simpler. Indeed, although this analysis follows the same basis-by-basis logic as in Zuidwijk’s approach, by Theorem \ref{th:g_cpiece}, each local part is determined by only three scalar parameters and is given by a homographic function rather than by a more general rational expression.

From an algorithmic perspective, since $g_c^\varepsilon(\lambda)$ can be viewed as a $\P_A(\lambda)$ problem, the methods developed for matrix perturbations can be directly reused. In particular, the bounding schemes proposed by \citet{MiftariEtAl2024LHSBounding} provide lower and upper bounds over an interval of $\lambda$, and thus a practical method for approximating $g^\varepsilon_c(\lambda)$, as illustrated in Figure \ref{LagrangianMethodExample}.

\begin{figure}[H]
    \centering
    \begin{minipage}{0.48\textwidth}
    \centering
        \includegraphics[width=1\linewidth]{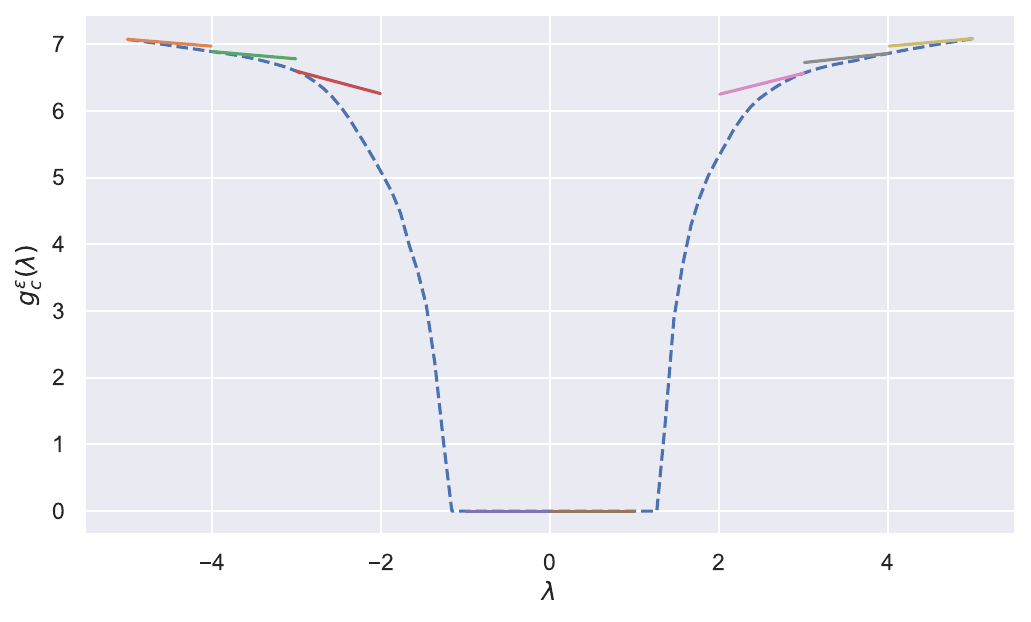}
    \end{minipage}
    \hfill
    \begin{minipage}{0.48\textwidth}
    \centering
        \includegraphics[width=1\linewidth]{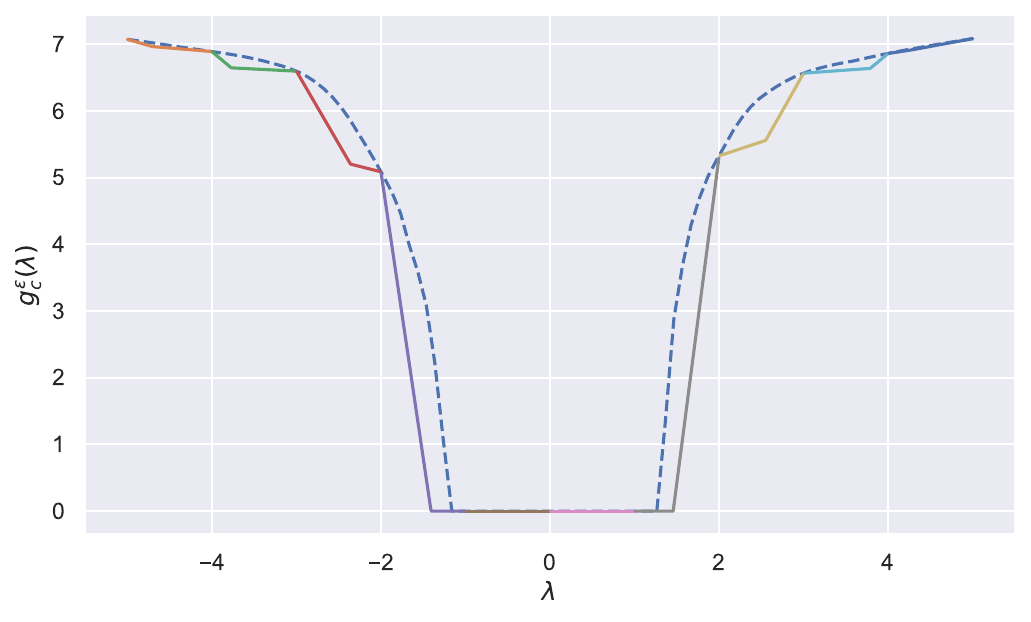}
    \end{minipage}
    \caption{Example of the solution using the bisegment coefficient Lagrangian method  \citep{MiftariEtAl2024LHSBounding}. The dotted blue curve shows the exact evolution of $g^\varepsilon_c(\lambda)$, while the colored segments correspond to the bounds constructed by the algorithm. On the left are the upper bounds, on the right are the lower bounds.}
    \label{LagrangianMethodExample}
\end{figure}

\subsection{Dual Form}\label{sub:dual_c}
\interp{In this case, the variable of interest is a dual one, such as the shadow price of the grid import constraint or the internal marginal value of electricity. The perturbation $\lambda\d$ again represents a change in economic conditions, such as a change in the electricity price applied to the microgrid.}{} 
In this section, $h_c(\lambda)$, the optimal value of the dual variables of interest for problem $\Q_c(\lambda)$ is analyzed. For a given vector of interest $ \boldsymbol{\varphi}$,

\begin{equation}
    \begin{aligned}
    h_c(\lambda) = \min_{\y}\quad & \boldsymbol{\varphi}^t\y\\
    \text{s.t.}\quad & \b^t \y \ge f_c(\lambda) \\
    & A^t \y \le (\c+\lambda \d)\\
    & \y \le 0.
    \end{aligned}
\end{equation}

Note that the constraint $\b^t \y \ge f_c(\lambda)$ involves the optimal primal value $f_c(\lambda)$. This is justified by \emph{strong duality} and because it is the maximum value of $\Q_c(\lambda)$. When both problems are feasible, the primal and dual optimal values coincide, so that $f_c(\lambda)$ can be considered as the optimal value of the dual problem $\Q_c(\lambda)$ for the same parameter $\lambda$. Moreover, observe that perturbing the cost vector $\c$ in the primal problem $\P_c(\lambda)$ corresponds to perturbing the right-hand side of the dual constraints $A^t\y\leq\c+\lambda\d$ in $\Q_c(\lambda)$.

\begin{theorem}\label{h_c_continu}
    The value function $h_c(\lambda)$ is continuous at every $\lambda \in \Lambda$, without adding a tolerance $\varepsilon$.
\end{theorem}

Theorem \ref{h_c_continu} is demonstrated later in this section, but it can be seen geometrically. Indeed, in the case of variable analysis on $\P_c(\lambda)$ (see Section \ref{sub:continue}), the objective vector $(\c+\lambda\d)$ varies linearly with $\lambda$.  The level hyperplanes $\bigl\{\x \mid (\c+\lambda\d)^t\x \le f_c(\lambda)\bigr\}$ therefore change their normal and rotate around the polytope. Conversely, in this case, looking at the dual, the hyperplane is on the constraints and translates parallel to them. More precisely, the realizable dual set depends on $\lambda$ through the right-hand side. The feasible polytope of problem $\Q_c(\lambda)$ can be written as $K_c^{D}(\lambda) := \bigl\{\, \y \in \mathbb{R}^m \mid \ y \le 0,\ A^t\y \le \c + \lambda \d \,\bigr\}$. Each inequality $\a_i^t\y \le c_i + \lambda d_i$ retains the same normal vector $\a_i$ and therefore defines a support hyperplane that translates parallel to itself when $\lambda$ varies. Furthermore, the constraint $\b^t\y \ge f_c(\lambda)$ does not pivot either. Indeed, its normal vector is $\b$, which is independent of $\lambda$, only the threshold $f_c(\lambda)$ varies. This means that geometrically, the feasible polytope $ K_c^D(\lambda)$ deforms by translation of the faces (rather than by rotation of the objective). In this context, the optimal value typically evolves continuously, and the basic changes correspond to moments when a constraint becomes active/inactive. There can be no switching due to the rotation of the objective, which does not produce discontinuity. Figure \ref{fig:casetranslation} illustrates the geometric intuition of the dual case.

\begin{figure}[h]
    \centering
    \begin{subfigure}[c]{0.62\linewidth}
        \centering
        \begin{subfigure}[c]{0.45\linewidth}
            \centering
            \includegraphics[scale=0.35]{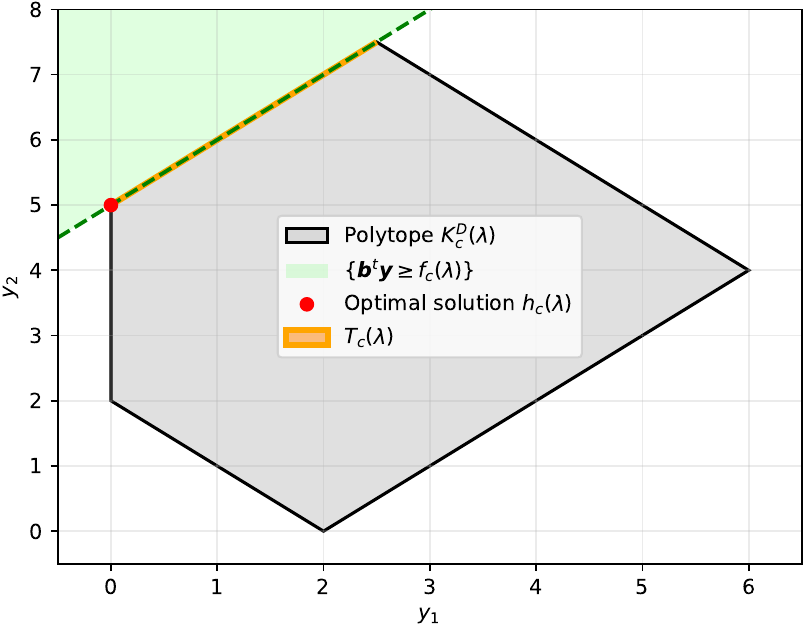}
        \end{subfigure}\hfill
        \begin{subfigure}[c]{0.45\linewidth}
            \centering
            \includegraphics[scale=0.35]{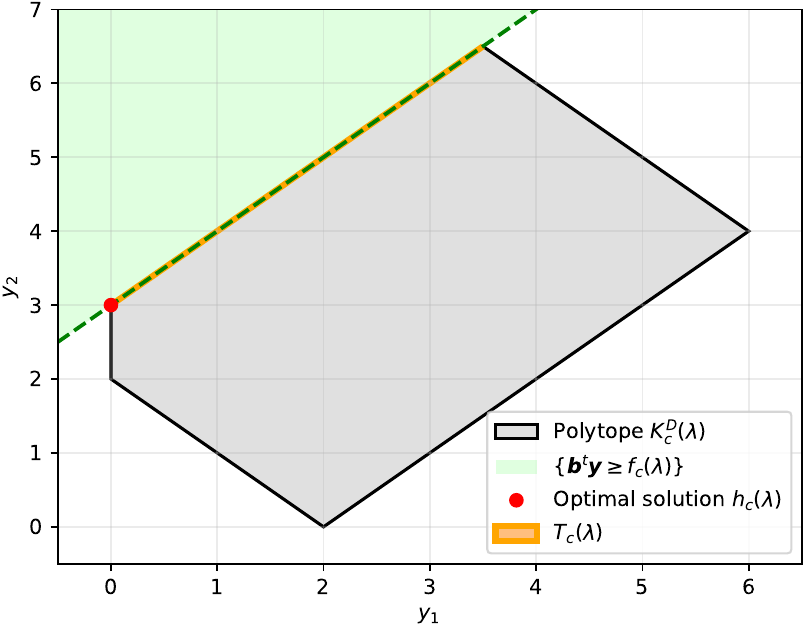}
        \end{subfigure}
        \caption{The dual feasible polytope $K_c^D(\lambda)$ and $\b^t\y \geq f_c(\lambda)$ are shown for $\lambda_1 = 2$  and $\lambda_2 = 6$. Between the two values, each constraint boundary translates parallel to itself, displacing the set $T_c(\lambda)$ and the selected solution $h_c(\lambda)$.}
    \end{subfigure}
    \hfill
    \begin{subfigure}[c]{0.32\linewidth}
        \centering
        \includegraphics[scale=0.35]{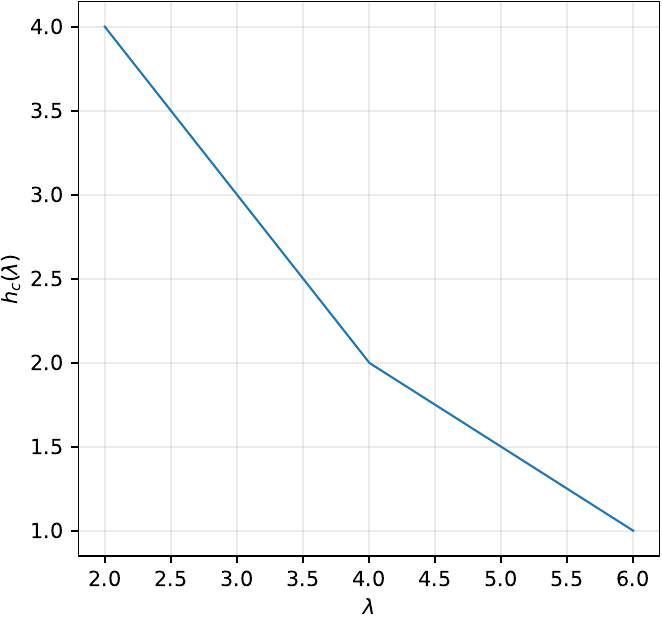}
        \caption{Representation of the value function $h_c(\lambda)$.}
    \end{subfigure}
    \caption{Geometric illustration of the $h_c(\lambda)$ in the dual problem $\Q_c(\lambda)$. Unlike the primal case (see Section \ref{sub:primal_c}) where the hyperplane of the objective $\{\x\mid (\c+\lambda \d)^t \x \le f_c(\lambda)\}$ rotates around the feasible polytope as $\lambda$, perturbing $\c$ in the dual shifts the right-hand side of the constraints $ \ A^t\y \le \c + \lambda \d$. Each constraint translates parallel to itself, and the constraint $\b^t\y \geq f_c(\lambda)$ also translates, since its normal $\b$ is independent of $\lambda$. This translational deformation of $K_c^D(\lambda)$ underlies the continuity of $h_c(\lambda)$ without the addition of tolerance $\varepsilon$.}
    \label{fig:casetranslation}
\end{figure}

By Theorem \ref{th:concave}, $f_c(\lambda)$ is a continuous, concave, piecewise linear function of $\lambda$ and can be represented as $f_c(\lambda)=\min \{\alpha_i \lambda + \beta_i \mid i \in 1...k \}, \ \forall \lambda \in \Lambda$. In the primal analysis, constraints of the form $(\c+\lambda \d)^t \x \le f_c(\lambda)$ can be rewritten as a finite family of linear inequalities since $t \le \min u_i \Longleftrightarrow t \le u_i\ \ \forall i$. In the dual problem, however, $f_c(\lambda)$ appears in the reverse inequality $\b^t \y \ge f_c(\lambda)$, and one cannot apply the same transformation globally. Indeed,  $t \ge \min u_i \ \centernot\Longleftrightarrow t \ge u_i\ \ \forall i$. Consequently, the dual analysis is performed interval by interval. 
As $f_c(\lambda)$ is a continuous concave piecewise linear function, there exists a finite partition of the parameter domain: $\Lambda=\bigcup_{j=1}^p I_j$ where $I_j=[\lambda_j,\lambda_{j+1}]$ for $p$ piecewise of $f_c(\lambda)$. On each interval $I_j$ a single affine piece is active $f_c(\lambda)=\alpha_j\lambda+\beta_j, \ \forall \lambda\in I_j$.
\begin{observation}\label{Obs:h_c_piece}
    Given  $f_c(\lambda)= \min \{ \alpha_j \lambda + \beta_j \mid j \in 1...p \}$, let an interval $I_j\in \Lambda$. For any $\lambda\in I_j$, the constraint $\b^t\y \ge f_c(\lambda)$ reduces to the linear constraint $\b^t\y \ge \alpha_j\lambda+\beta_j$, and define the corresponding subproblem \begin{equation}\label{eq:hc_piece}
        \begin{aligned}
        h_c^{(j)}(\lambda)=\min_{\y}\quad & \boldsymbol{\varphi}^t\y\\
        \text{s.t.} \quad
        & \b^t \y \ge \alpha_j \lambda + \beta_j\\
        & A^t \y \le \c+\lambda \d\\
        & \y \le 0,
        \end{aligned}
    \end{equation}
    By construction, $h_c(\lambda)=h_c^{(j)}(\lambda)$ for all $\lambda\in I_j$. On each interval, the problem is a linear program with right-hand sides varying affinely with $\lambda$, so $h_c^{(j)}(\lambda)$ has the same qualitative behavior as the value function $f_b(\lambda)$ of problem $\P_b(\lambda)$: it is a continuous piecewise linear convex function.
\end{observation}

Over the entire domain, $h_c(\lambda)$ is obtained by repeating this analysis on each affine part of $f_c(\lambda)$. Therefore, $h_c(\lambda)$ is a piecewise function composed of piecewise linear convex functions. The breakpoints can come from (i) the breakpoints of $f_c(\lambda)$ (when the pair $(\alpha,\beta)$ changes) or (ii) changes in the optimal basis in the perturbation subproblem of the RHS considered over a given interval. In particular, once the decomposition of $f_c(\lambda)$ into pieces is known, the calculation of $h_c(\lambda)$ reduces to solving a sequence of parametric linear programming subproblems, which can be efficiently solved. Nevertheless, Theorem \ref{th:murty} shows that, the number of breakpoints can be exponential in the worst case. 
Therefore, an exact "piece-by-piece" approach may require an exponential number of subproblems to be solved in the worst case. Figure \ref{fig:h_c_representation} illustrates the evolution of $h_c(\lambda)$ as a function of $\lambda$, while the dotted vertical lines mark the breakpoints of $f_c(\lambda)$, i.e., the values of $\lambda$ where the optimal value changes piecewise. It can be seen that the convex linear functions are between the pieces of $f_c(\lambda)$.\\

\begin{figure}[h]
    \centering    \includegraphics[width=0.6\linewidth]{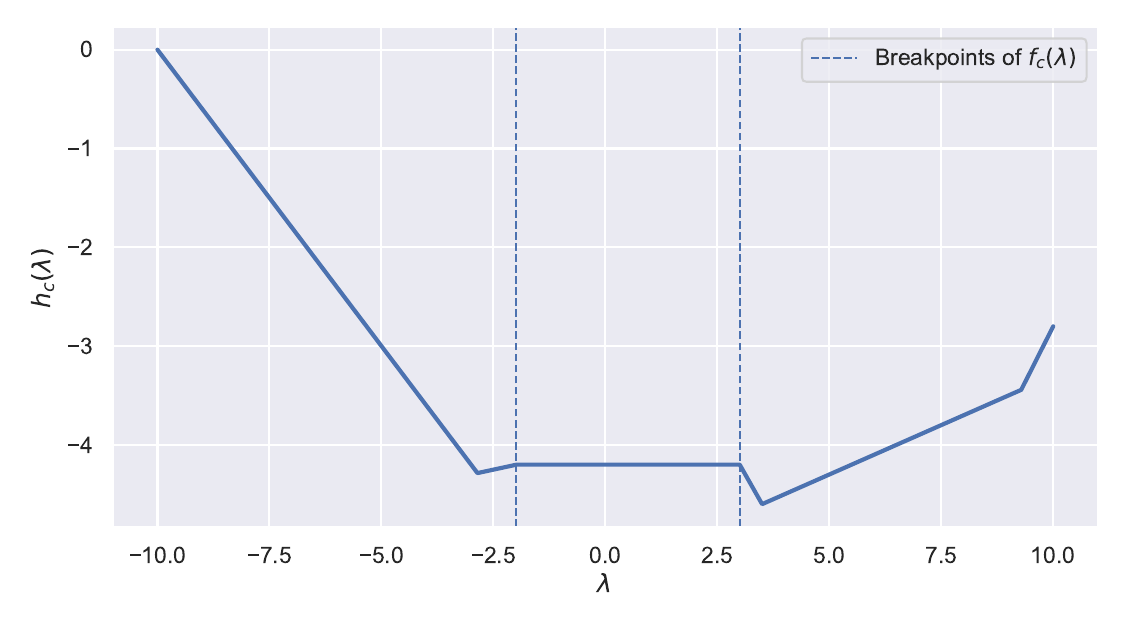}
    \caption{Example of $h_c(\lambda)$, a continuous piecewise function composed of piecewise linear convex functions.}
    \label{fig:h_c_representation}
\end{figure}

\begin{proof}[Proof of Theorem \ref{h_c_continu}]
    Let us prove that $h_c(\lambda)$ is continuous. According to Observation \ref{Obs:h_c_piece},
     $h_c(\lambda)=h_c^{(j)}(\lambda)$ for all $\lambda\in I_j$. It suffices to demonstrate that: (i) $h_c^{(j)}(\lambda)$ is continuous on $I_j$, and (ii) the pieces are continuously connected at the breakpoints. Previously, it has been shown that $h_c^{(j)}(\lambda)$ is analogous to $f_b(\lambda)$, which is continuous. Let us demonstrate that at the breakpoints, the pieces are continuously reconnected.\\
    Let $\bar\lambda=\lambda_{j+1}$ be a breakpoint of $f_c(\lambda)$.
    Since $f_c(\lambda)$ is continuous, $\alpha_j \bar\lambda+\beta_j = f_c(\bar\lambda)=\alpha_{j+1}\bar\lambda+\beta_{j+1}$. At $\lambda=\bar\lambda$, $h_c^{(j)}(\bar\lambda)$ and $h_c(\lambda)=h_c^{(j)}(\bar\lambda)$ impose exactly the same constraint $\b^t\y\ge f_c(\bar\lambda)$, so describe the same problem. Thus
    $h_c^{(j)}(\bar\lambda)=h_c^{(j+1)}(\bar\lambda)=h_c(\bar\lambda)$. Since $h_c^{(j)}$ is continuous on $I_j$ and $h_c^{(j+1)}$ is continuous on $I_{j+1}$, so 
    $\lim_{\lambda\rightarrow \bar\lambda} h_c(\lambda) = h_c^{(j)}(\bar\lambda)= h_c(\bar\lambda), \qquad \lim_{\lambda\rightarrow  \bar\lambda} h_c^{(j+1)}(\lambda) = h_c^{(j+1)}(\bar\lambda) = h_c(\bar\lambda)$. $h_c(\lambda)$ is continuous in $\bar\lambda$, since there is only a finite number of breakpoints, $h_c(\lambda)$ is continuous in $\Lambda$.
\end{proof}

\subsection*{Illustration (Example \ref{example_microgrid} : Microgrid)}
In this subsection, let us return to Example \ref{example_microgrid} to illustrate the properties of cost vector perturbations $\c$ defined in Section \ref{section:c}. Figure \ref{fig:microgrid_c} shows the functions $f_c(\lambda)$, $g_c(\lambda)$, $g^\varepsilon_c(\lambda)$, and $h_c(\lambda)$ obtained for the microgrid example when $\lambda$ pertubs the grid electricity price.\\
The upper subfigure is consistent with Theorem \ref{th:concave}. Indeed, the optimal value function $f_c(\lambda)$ is continuous, concave, and piecewise linear over the parameter interval. The middle subfigure represents the primal $g_c(\lambda)$, which, in this case, is associated with the minimum total import over the set of optimal solutions. As predicted by the analysis in Section \ref{sub:continue}, $g_c(\lambda)$ is discontinuous. Therefore, we have an illustration of the case where adding a tolerance $\varepsilon$ leads to the homographic form of $g^\varepsilon_c(\lambda)$, discussed in Section \ref{sub:reformulation}. In contrast,  the lower subfigure shows that the dual $h_c(\lambda)$ depends continuously on $\lambda$. Its profile is the continuous piecewise function composed of piecewise linear convex functions as demonstrated in Section \ref{sub:dual_c}. Here, $h_c(\lambda)$ represents the minimum shadow price of electricity associated with constraints related to the network import limit that can be reached while remaining at optimal costs during the perturbation of vector $\c$.

\begin{figure}[h]
    \centering
    \includegraphics[width=0.75\linewidth]{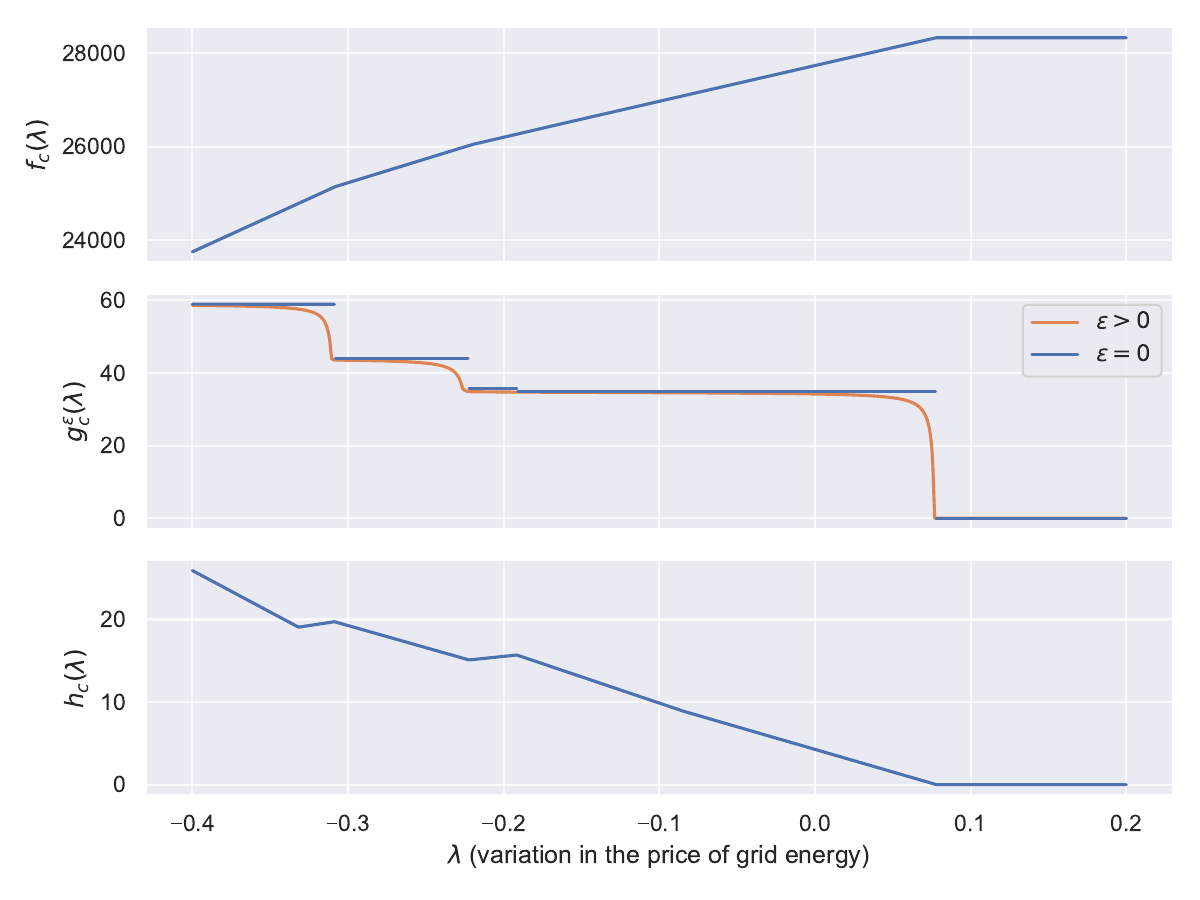}
    \caption{Sensitivity profiles for the cost perturbation $\P_c(\lambda)$ (variation in the price of grid energy) in Example \ref{example_microgrid}. Top: optimal value function $f_c(\lambda)$, continuous, concave, and piecewise linear. Middle: primal $g_c(\lambda)$ and $g^\varepsilon_c(\lambda)$, defined as the minimum total grid import over the set of optimal solutions. Bottom: dual $h_c(\lambda)$, defined as the minimum average shadow price associated with the grid import limit constraints.}
    \label{fig:microgrid_c}
\end{figure}
\section{Sensitivity, when modifying $\b$}
\label{section:b}

\interp{In this case, the interest $\boldsymbol{\varphi}^t\x$ is associated with the decision variable of the microgrid, for example, the installed battery capacity or the total electricity imported from the grid. The perturbation $\lambda \d$ in the right-hand side corresponds to a change in a demand or an available resource, such as electricity demand or the maximum import capacity of the grid connection.}{ Primal}
\interp{For the dual variables, one may choose $\boldsymbol{\varphi}^t\y$ to study the marginal price of electricity or the value of relaxing a capacity constraint. The perturbation $\lambda \d$ remains a right-hand side perturbation, corresponding, for example, to a variation in demand or in the available import capacity of the external grid.}{ Dual}

Now that the analysis of the variation of the objective function $\c$ is complete, the variation of the right-hand side term $\b$ can be analyzed. The analysis of right-hand side $\b$ perturbations follows a structure analogous to the one developed in Section \ref{section:c} for cost vector $\c$ perturbations, but with a primal-dual asymmetry that is exactly reversed. This reversal is a direct consequence of the duality correspondence between the two problems. Perturbing $\b$ in  $\P_b(\lambda)$ corresponds to perturbing the objective vector of dual $\Q_b(\lambda)$, in the same way that perturbing $\c$ in $\P_c(\lambda)$ corresponds to perturbing the right-hand side of the dual constraints in $\Q_c(\lambda)$. In this case, the following well-known theorem holds, which is very close to Theorem \ref{th:concave} \citep[see, for example,][]{berkelaar1997optimal, Bertsimas1997IntroductionTL}.

\begin{theorem}\label{th:convexe}
Let $\Lambda\subseteq\mathbb{R}$ be an interval such that the problem $\P_b(\lambda)$ has a finite optimal value for every $\lambda\in\Lambda$. Then $f_b(\lambda)$, the optimal value function of $\P_b(\lambda)$ is continuous, convex, and piecewise linear on $\Lambda$.
\end{theorem}

As a consequence, the primal and dual selection functions $g(\lambda)$ and $h(\lambda)$ described throughout Section \ref{section:c} are here consistently exchanged. Indeed, every property established for $g_c^\varepsilon(\lambda)$ holds for $h_b^\varepsilon(\lambda)$, and every property established for $h_c(\lambda)$ holds for $g_b(\lambda)$.

Specifically, $g_b(\lambda)$, as $h_c(\lambda)$ (see Section \ref{sub:dual_c}), is a piecewise function composed of piecewise linear convex functions, continuous even when $\varepsilon = 0$, obtained by applying the same interval analysis to each piece $I_j$ of $f_b(\lambda)$. Conversely, $h_b^\varepsilon(\lambda)$ has the same properties as $g_c^\varepsilon(\lambda)$ in Section \ref{sub:primal_c}, i.e, it may be discontinuous for $\varepsilon = 0$.  However, for $\varepsilon > 0$, it is a piecewise homographic continuous function and can be reformulated as a parametric problem of the type $\P_A(\lambda)$, making bounding schemes such as those of \citet{MiftariEtAl2024LHSBounding} directly applicable.

The proofs of all these results follow the same reasoning as those of Lemmas \ref{T:USC} and \ref{T:LSC}, Theorems \ref{Berge_maximum}, \ref{T:P}, and \ref{th:g_cpiece}, and Observation \ref{obs:equivalentA} in Section \ref{section:c}, applied to the case of a perturbation on the right-hand side $\b$. More details are presented in Appendix \ref{appendix:B}.

\subsection*{Illustration (Example \ref{example_microgrid}: Microgrid)}
In this subsection, Example \ref{example_microgrid} is revisited to illustrate the properties of the right-hand side perturbations $\b$ introduced in Section \ref{section:b}. Figure \ref{fig:microgrid_b} highlights the functions $f_b(\lambda)$, $g_b(\lambda)$, $h_b(\lambda)$ and $h^\varepsilon_b(\lambda)$ obtained for the microgrid example when $\lambda$ represents a perturbation of the grid import limit.
The upper subfigure shows the optimal value function $f_b(\lambda)$, which is continuous, convex, and piecewise linear. The middle subfigure represents the primal  $g_b(\lambda)$, the minimum installed battery capacity in all optimal solutions. The bottom subfigure represents the dual  $h^\varepsilon_b(\lambda)$, which is the minimum average internal marginal value of electricity in all optimal dual solutions.
It should be noted that, in the case of right-hand side perturbations, the asymmetry between primal sensitivity and dual sensitivity is reversed compared to the case of cost vector $\c$ perturbations. 
In fact, in this case, it follows that $g_b(\lambda)$ is a piecewise function composed of piecewise linear convex functions. $h^\varepsilon_b(\lambda)$ is a discontinuous function for $\varepsilon = 0$ and is a continuous piecewise homographic function for $\varepsilon > 0$. 
\begin{figure}[h]
    \centering
    \includegraphics[width=0.75\linewidth]{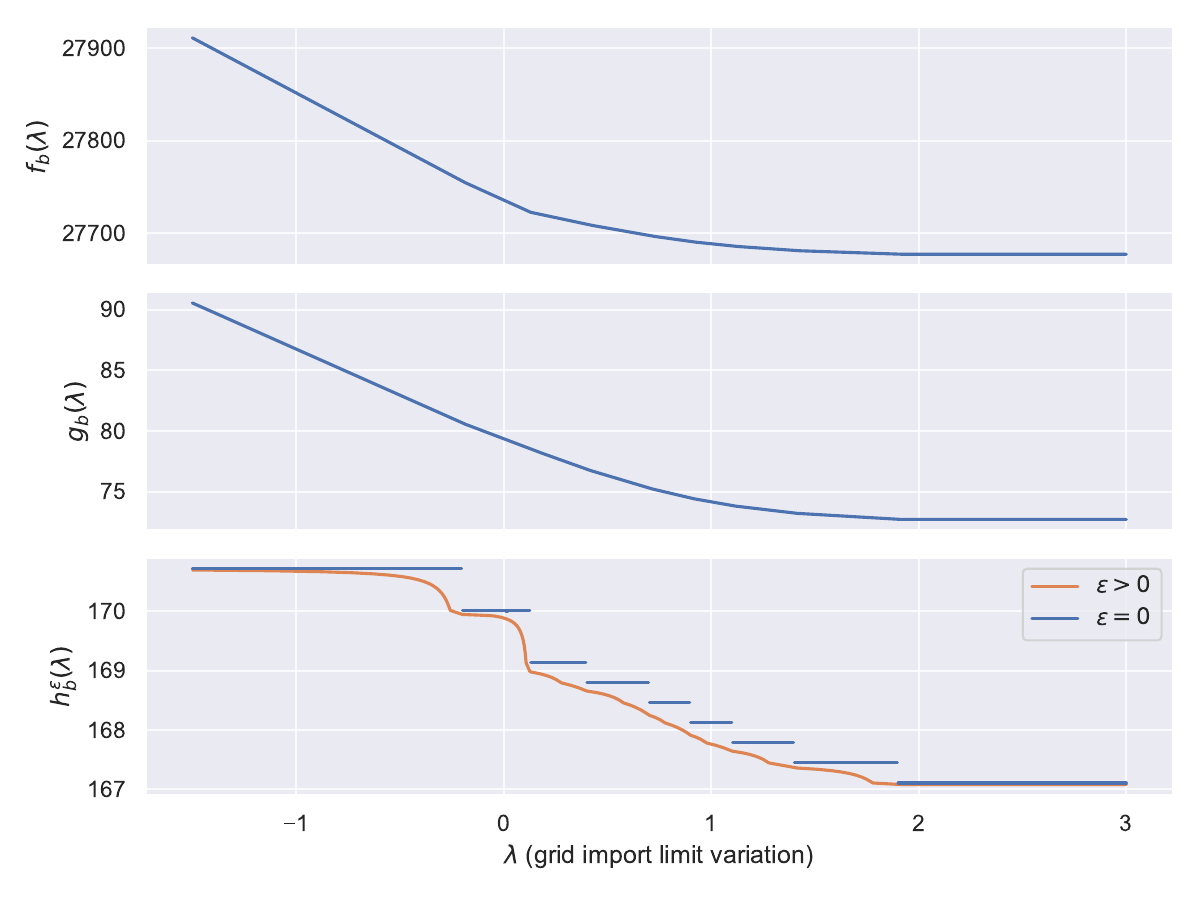}
    \caption{Sensitivity profiles for the right-side perturbation $\P_b(\lambda)$ in Example \ref{example_microgrid}. Top: value function $f_b(\lambda)$, continuous, convex, and piecewise linear. Middle: primal $g_b(\lambda)$, defined, in this case,  as the minimum installed battery capacity across all optimal solutions. Bottom: dual $h_b(\lambda)$ and $h^\varepsilon_b(\lambda)$, defined here as the minimum average marginal value of electricity across all optimal dual solutions.}
    \label{fig:microgrid_b}
\end{figure}

\section{Sensitivity, when modifying $A$}\label{section:A}

\interp{In this case, the primal variable of interest may be the installed battery capacity, the installed PV capacity, or the total imported energy, while the perturbation $\lambda D$ modifies a technical coefficient in the model, such as the battery efficiency or the conversion from installed PV capacity to hourly PV production.}{ Primal}
\interp{For the dual problem, one may study the marginal value of electricity, storage capacity, or grid access under a perturbation of the technical coefficients. In this case, $\lambda D $ represents a change in the physical performance of the microgrid components.}{ Dual}

Finally, consider perturbations of the constraint matrix $A$, problems of the form $\P_A(\lambda)$ where the parameter $\lambda$ appears on the left-hand side via $A+\lambda D$. This case is generally more delicate than perturbations of $\c$ or $\b$. Indeed, the variation of $A$ modifies the geometry of the feasible region itself, can induce basic singularities, and can generate more irregular parametric behavior.

Let $f_A(\lambda)$ denote the optimal value of the primal problem $\P_A(\lambda)$:
\begin{equation}\label{eq:fA_primal}
    \begin{aligned}
    f_A(\lambda) = \min_{\x}\quad & \c^{t}\x\\
    \text{s.t.}\quad & (A+\lambda D)\x \le \b\\
    & \x \ge 0.
    \end{aligned}
\end{equation}

By strong duality (whenever $\P_A(\lambda)$ is feasible and bounded), the dual optimal value coincides with the primal one. According to \citet{Zuidwijk2005LinearParametric} and \eqref{eq:zuidwijk_rational}, the optimal value function $f_A(\lambda)$ is a piecewise rational function of $\lambda$ (i.e., each local expression is a quotient of polynomials). Moreover, by Theorem \ref{Berge_maximum} and Observation \ref{obs:equivalentA}, $f_A(\lambda)$ may be discontinuous. Indeed, since $g_c(\lambda)$ can be written as an $A+\lambda D$ problem and has discontinuities, this means that in the worst case, $f_A(\lambda)$ also has discontinuities. This already contrasts with the cases studied for $\P_c(\lambda)$ and $\P_b(\lambda)$, where the value functions are continuous and piecewise affine.
As in the previous sections, the interest is not only in the optimal value $f_A(\lambda)$ but also in the selection of representative (quasi-)optimal solutions through a linear criterion $\boldsymbol{\varphi}$. Consider the following primal and dual selection problems :

\begin{equation}
    \begin{minipage}[c]{0.48\linewidth}
        \centering
        \vspace{0pt}
        $\begin{aligned}
        g^\varepsilon_A(\lambda) = \min_{\x} \quad & \boldsymbol{\varphi}^t\x\\
      \text{s.t.} \quad & \c^{t}\x \leq f_A(\lambda)+\varepsilon \\ & (A+\lambda D)\x\le \b \\
         & \x \geq 0 
        \end{aligned}$
    \end{minipage}
    \hfill
    \begin{minipage}[c]{0.48\linewidth}
        \vspace{0pt}
        \centering
        $\begin{aligned}
            h^\varepsilon_A(\lambda) = \min_{\y}\quad & \boldsymbol{\varphi}^t\y\\
            \text{s.t.}\quad & \b^t \y \ge f_A(\lambda)-\varepsilon \\
            & (A+\lambda D)^t \y \le \c\\
            & \y \le 0.
        \end{aligned}$
    \end{minipage}
\end{equation}

\begin{observation}
    Unlike $g_c^\varepsilon(\lambda)$, for which Theorem \ref{Berge_maximum} guarantees continuity for every $\varepsilon>0$, the function $g_A^\varepsilon(\lambda)$ may have discontinuities even when a tolerance $\varepsilon>0$ is introduced.
\end{observation}

\begin{example}\label{exampleg_A}
    Consider the following problem of $\P_A(\lambda)$:
    \begin{equation*}
        \begin{aligned}
            f_A(\lambda) =&  \max_{x} \; x \\
            \text{s.t.}\quad
            & \left\{
            \begin{aligned}
                x &\le 5\\
                \lambda x & \ge 0,
            \end{aligned}
            \right.
        \end{aligned}
    \end{equation*}
    \begin{enumerate}
        \item \textbf{Case $\lambda<0$} : the constraint $\lambda x\ge0$ forces $x\le0$, so $f_A(\lambda)=0$.
        \item \textbf{Case $\lambda\ge0$} : the constraint $\lambda x\ge0$ is redundant ($\lambda>0$) or trivial ($\lambda=0$), leaving $x\le5$ as the only active bound, hence $f_A(\lambda)=5$.
    \end{enumerate}
Combining both cases, we obtain
\begin{equation*}
    f_A(\lambda) =
    \begin{cases}
        0 & \text{if } \lambda<0\\
        5 & \text{if }\lambda\ge0.
    \end{cases}
\end{equation*}
Taking $\boldsymbol{\varphi}=\begin{pmatrix}
    1
\end{pmatrix}$,
\begin{equation*}
        \begin{aligned}
            g^\varepsilon_A(\lambda) =&  \max_{x} \; x \\
            \text{s.t.}\quad
            & \left\{
            \begin{aligned}
                x &\le 5\\
                \lambda x & \ge 0\\
                x &\geq f_A(\lambda) - \varepsilon.
            \end{aligned}
            \right.
        \end{aligned}
    \end{equation*}
\begin{enumerate}
    \item \textbf{Case $\lambda<0$} : the additional constraint implies that $x\ge f_A(\lambda)-\varepsilon=-\varepsilon$, so the feasible region is $-\varepsilon\le x\le0$. The maximum is attained at $x=0$, hence $g^\varepsilon_A(\lambda)=0$.
    \item \textbf{Case $\lambda\ge0$} : the additional constraint is $x\ge f_A(\lambda)-\varepsilon=5-\varepsilon$, leaving the feasible region $5-\varepsilon\le x\le5$. The maximum is attained at $x=5$, hence $g^\varepsilon_A(\lambda)=5$.
\end{enumerate}
Hence,
\[
g_A^\varepsilon(\lambda) =
\begin{cases}
0 & \lambda\leq0\\
5 & \lambda>0,\\
\end{cases}
\]
which, for any $\varepsilon\ge0$, has a discontinuity for $\lambda=0$. 

\end{example}
\vspace{3mm}

The discontinuity mentioned above is one of several aspects of the difficulties posed by the perturbations of $\P(A)$. A second, challenging problem concerns the reformulation of the selection problems themselves.
In Sections \ref{section:c} and \ref{section:b}, a key simplification comes from the piecewise affine structure of the parametric value functions, which allows the quasi-optimality constraints to be transformed into linear systems, either globally (via min/max manipulations) or interval by interval (where the value was affine). These reformulations do not directly extend to $\P_A(\lambda)$. Since $f_A(\lambda)$ is only known to be piecewise rational, limiting the analysis to a single "piece" does not automatically simplify the model.

Given the above difficulties, an alternative is to work with efficiently computable bounds of $f_A(\lambda)$. In particular, bounding schemes such as those proposed by \cite{MiftariEtAl2024LHSBounding} construct valid upper and lower bounds (possibly dependent on $\lambda$) for $f_A(\lambda)$ over an interval, generally in piecewise affine form. These bounds can then be used to construct approximations of the selection problems. For example, if $UB_{f_A(\lambda)}\ge f_A(\lambda)$ is an upper bound, the following model can be considered
\begin{equation}
    \begin{aligned}
    g^\varepsilon_A(\lambda) = \min_\x \quad &\boldsymbol{\varphi}^t\x\\
    \text{s.t.}\quad &\c^t \x \le  UB_{f_A(\lambda)}\\& (A+\lambda D)\x \le \b\\
    & \x \ge 0,
    \end{aligned}
\end{equation}
which restores a linear structure whenever $UB_{f_A(\lambda)}$ is piecewise affine. Symmetrically, for the dual selection problem, a lower bound gives a relaxed constraint $\b^t \y \ge LB_{f_A(\lambda)}$. Overall, perturbations of $A$ are the most complex case in this study. Indeed, even the value function $f_A(\lambda)$ is only piecewise rational and can be discontinuous, and decision-making inherits this complexity.

\subsection*{Illustration (Example \ref{example_microgrid}: Microgrid)}

Figure \ref{fig:microgrid_A} illustrates the functions $f_A(\lambda)$ and $g^\varepsilon_A(\lambda)$ obtained for the microgrid example when $\lambda$ disturbs a technical coefficient of the constraint matrix $(\P_A(\lambda))$. It is a perturbation on the efficiency of PVs. The upper subfigure shows the optimal value function $f_A(\lambda)$. Unlike the cases of cost and right-hand side perturbations, the matrix perturbation parameter leads to more irregular parametric behavior. As discussed in Section \ref{section:A}, the function $f_A(\lambda)$ is piecewise rational and may potentially have discontinuities. The lower subfigure represents the primal $g_A(\lambda)$ and $g_A^\varepsilon(\lambda)$, which is associated with the minimum total grid import over the set of optimal solutions. Figure \ref{fig:microgrid_A} illustrates the more complex behavior of the matrix perturbation case, with irregular parts and visible discontinuities, in accordance with the analysis developed in Section \ref{section:A}.

\begin{figure}[h]
    \centering
    \includegraphics[width=0.75\linewidth]{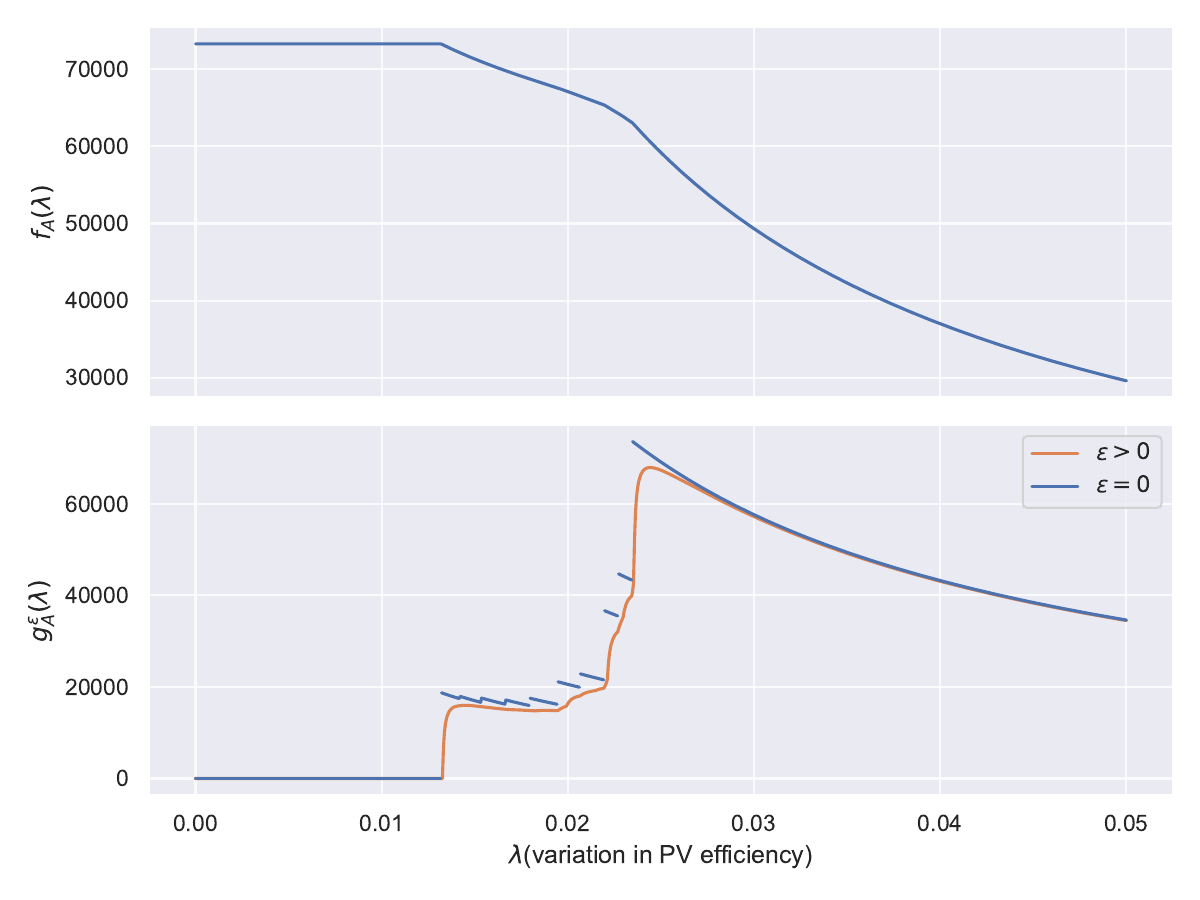}
    \caption{Sensitivity profiles for the matrix perturbation $\P_A(\lambda)$ in Example \ref{example_microgrid}. Top: optimal value function $f_A(\lambda)$. Bottom: primal $g_A(\lambda)$ and $g^\varepsilon_A(\lambda)$, defined here as the minimum value of total grid import over the set of optimal solutions.}
    \label{fig:microgrid_A}
\end{figure}





\section{Conclusion}\label{section:conclu}
In conclusion, this work studies the sensitivity analysis for linear programs under linear perturbations of (i) the objective function (via $\c$), (ii) the RHS (via $\b$), and (iii) the constraint matrix (via $A$). Rather than focusing on the optimal objective value, we propose an approach focused on decisions, consisting of selecting, via a linear criterion $\boldsymbol{\varphi}$, representative solutions from the set of optimal or near-optimal solutions. 

The results highlight a primal-dual asymmetry depending on the type of perturbation. When $\c$ is perturbed, the primal selection function $g_c(\lambda)$ may be discontinuous, as the optimal face can abruptly change dimension at critical values of $\lambda$. The introduction of a tolerance $\varepsilon>0$, however small, is sufficient to make $g^\varepsilon_c(\lambda)$ continuous. This near-optimal requirement is practically acceptable, since data uncertainty, measurement noise, and solver tolerances already prevent exact optimality in applications. Furthermore, for each interval where the optimal basis remains unchanged, $g^\varepsilon_c(\lambda)$ is a piecewise homographic function determined by only three scalar parameters, and the overall problem can be reformulated as a parametric problem of type $\P_A(\lambda)$. In contrast, the dual selection function $h_c(\lambda)$ is continuous without requiring any tolerance and is a piecewise function composed of piecewise linear convex functions.  When $\b$ is perturbed, the situation is reversed; $g_b(\lambda)$ shares the properties of $h_c(\lambda)$ and $h_b^\varepsilon(\lambda)$ those of $g_c^\varepsilon(\lambda)$.

Perturbations of the constraint matrix $A$ are distinguished by increased complexity. The value function $f_A(\lambda)$ is only piecewise rational and may itself be discontinuous, and the selection functions inherit this complexity. This case remains the least understood of the three and the one where computational tools, such as the bounding schemes of \citet{MiftariEtAl2024LHSBounding}, are most needed.

\section*{CRediT authorship contribution statement}
\noindent
\textbf{Baptiste Istace:} Conceptualization, Investigation, Methodology, Software, Visualisation, Writing - original draft;
\textbf{Guillaume Derval:} Conceptualization, Methodology, Software, Supervision, Writing - original draft;
\textbf{Bardhyl Miftari:} Conceptualization, Software, Supervision, Writing- review \& editing;
\textbf{Quentin Louveaux:} Conceptualization, Methodology, Supervision, Writing - review \& editing.

\bibliographystyle{elsarticle-harv}
\bibliography{bib}

\newpage
\appendix
\section{Main notation}\label{appendix:A}

This appendix lists the main notations used throughout this paper. For ease of reference, Table \ref{tab:notation} summarizes the notations related to primal and dual parametric problems, feasibility sets, sets of optimal solutions, value functions, selection functions, and perturbation parameters.

\begin{longtable}{@{}p{0.18\textwidth} p{0.76\textwidth}@{}}
\toprule
\textbf{Notation} & \textbf{Meaning} \\
\midrule
\endfirsthead

\toprule
\textbf{Notation} & \textbf{Meaning} \\
\midrule
\endhead

\bottomrule
\endfoot

\multicolumn{2}{@{}l}{\textit{Base primal--dual}}\\
$\P$   & Reference primal linear programming problem. \\
$\Q$   & Dual problem associated with $\P$. \\
$n$    & Number of variables in the primal problem.\\
$m$    & Number of constraints in the primal problem.\\
$\x$   & Vector of primal variables, with $\x\in\mathbb{R}^n$. \\
$\y$   & Vector of dual variables, with $\y\in\mathbb{R}^m$. \\
$\c$   & Cost coefficient vector in the primal objective, with $\c\in\mathbb{R}^n$. \\
$\b$   & Right-hand side vector of the primal constraints, with $\b\in\mathbb{R}^m$. \\
$A$    & Constraint matrix of the primal problem, with $A\in\mathbb{R}^{m\times n}$. \\
$K$    & Primal feasible region, defined by
         $K=\{\x\in\mathbb{R}_+^n \mid A\x\le \b\}$. \\
$K^D$  & Dual feasible region, defined by
         $K^D=\{\y\in\mathbb{R}^m \mid A^t \y\le \c, \y\le 0\}$. \\

\multicolumn{2}{@{}l}{\textit{Parametric problems}}\\
$\P_c(\lambda)$ & Primal parametric problem with perturbation of the cost vector $\c$. \\
$\P_b(\lambda)$ & Primal parametric problem with perturbation of the right-hand side $\b$. \\
$\P_A(\lambda)$ & Primal parametric problem with perturbation of the constraint matrix $A$. \\
$\Q_c(\lambda)$ & Dual parametric problem associated with $\P_c(\lambda)$. \\
$\Q_b(\lambda)$ & Dual parametric problem associated with $\P_b(\lambda)$. \\
$\Q_A(\lambda)$ & Dual parametric problem associated with $\P_A(\lambda)$. \\[0.3em]

\multicolumn{2}{@{}l}{\textit{Optimal value functions}}\\
$f_c(\lambda)$ & Optimal value function of $\P_c(\lambda)$. \\
$f_b(\lambda)$ & Optimal value function of $\P_b(\lambda)$. \\
$f_A(\lambda)$ & Optimal value function of $\P_A(\lambda)$. \\[0.3em]

\multicolumn{2}{@{}l}{\textit{Feasible sets}}\\
$K_c(\lambda)$   & Feasible set of $\P_c(\lambda)$; here $K_c(\lambda)=K$. \\
$K_b(\lambda)$   & Feasible set of $\P_b(\lambda)$, defined by
                   $K_b(\lambda)=\{\x\in\mathbb{R}_+^n \mid A\x\le \b+\lambda\d\}$. \\
$K_A(\lambda)$   & Feasible set of $\P_A(\lambda)$, defined by
                   $K_A(\lambda)=\{\x\in\mathbb{R}_+^n \mid (A+\lambda D)\x\le \b\}$. \\
$K_c^D(\lambda)$ & Dual feasible set under cost perturbation. \\
$K_b^D(\lambda)$ & Dual feasible set under right-hand-side perturbation; here
                   $K_b^D(\lambda)=K^D$. \\
$K_A^D(\lambda)$ & Dual feasible set under matrix perturbation. \\[0.3em]

\multicolumn{2}{@{}l}{\textit{Optimal solution sets}}\\
$S(\lambda)$     & Set of optimal primal solutions for a given value of $\lambda$. \\
$T(\lambda)$   & Set of optimal dual solutions for a given value of $\lambda$. \\
$S_c(\lambda)$   & Set of optimal primal solutions of $\P_c(\lambda)$. \\
$S_b(\lambda)$   & Set of optimal primal solutions of $\P_b(\lambda)$. \\
$S_A(\lambda)$   & Set of optimal primal solutions of $\P_A(\lambda)$. \\
$T_c(\lambda)$ & Set of optimal dual solutions of $\Q_c(\lambda)$. \\
$T_b(\lambda)$ & Set of optimal dual solutions of $\Q_b(\lambda)$. \\
$T_A(\lambda)$ & Set of optimal dual solutions of $\Q_A(\lambda)$. \\[0.3em]

\multicolumn{2}{@{}l}{\textit{Perturbation parameters}}\\
$\lambda$ & Scalar perturbation parameter. \\
$\Lambda$ & Parameter domain of $\lambda$. \\
$\d$ & Perturbation direction vector for $\c$ or $\b$. \\
$D$ & Perturbation direction matrix for the constraint matrix $A$. \\[0.3em]

\multicolumn{2}{@{}l}{\textit{Selection functions}}\\
$\boldsymbol{\varphi}$ & Vector defining variables of interest. \\
$g(\lambda)$ & Optimal value of the primal variables of interest, defined by $g(\lambda)=\min_{\x\in S(\lambda)} \boldsymbol{\varphi}^t\x$. \\
$h(\lambda)$ & Optimal value of the dual variables of interest, defined by $h(\lambda)=\min_{\y\in T(\lambda)} \boldsymbol{\varphi}^t\y$. \\
$g_c(\lambda)$ & Optimal value of the primal variables of interest under cost vector $\c$ perturbation. \\
$g_b(\lambda)$ & Optimal value of the primal variables of interest under RHS $\b$ perturbation. \\
$g_A(\lambda)$ & Optimal value of the primal variables of interest under matrix $A$ perturbation. \\
$h_c(\lambda)$ & Optimal value of the dual variables of interest under cost vector $\c$ perturbation. \\
$h_b(\lambda)$ & Optimal value of the dual variables of interest under RHS $\b$ perturbation. \\
$h_A(\lambda)$ & Optimal value of the dual variables of interest under matrix $A$ perturbation.  \\[0.3em]

\multicolumn{2}{@{}l}{\textit{Tolerance}}\\
$\varepsilon$ & Optimality tolerance used to define near-optimal solutions. \\
\caption{Main notation used in the paper}\label{tab:notation}\\
\end{longtable}

\newpage

\section{Details Sensitivity, when modifying $\b$}\label{appendix:B}

In this appendix, the asymmetry of the primal-dual between the cost vector perturbation $\c$ and the right-hand side $\b$ discussed in Section \ref{section:b} is detailed. Let $f_b(\lambda)$ be the optimal value function of the primal problem $\P_b(\lambda)$,

\begin{equation}\label{eq:fb_primal_dual}
\begin{aligned}
            f_b(\lambda) = \min_{\x}\quad & \c^{t}\x\\
            \text{s.t.}\quad & A\x \le \b + \lambda \d\\
            & \x \ge 0.
\end{aligned}
\end{equation}

\subsection{Primal form}
\label{sub:primal_b}

The primal variables are analyzed in the same way as the dual variables when the cost vector $\c$ varies, as shown in Section \ref{sub:dual_c}. Since $f_b(\lambda)$ is the minimum of $\c^t\x$ over the feasible set, the constraint $\c^t\x = f_b(\lambda)$ can be rewritten as $\c^t\x \leq f_b(\lambda)$. Let introduce $g_b(\lambda)$ such that 
\begin{equation}
    \begin{aligned}
        g_b(\lambda) &= \min_\x \quad  \boldsymbol{\varphi}^t\x\\
        \text{s.t. }&  A\x \leq \b + \lambda \d \\
        &\c^t\x \leq f_b(\lambda)\\
         & x \geq 0.
    \end{aligned}
\end{equation}

According to Theorem \ref{th:convexe}, $f_b(\lambda)$ is the maximum of multiple linear functions of $\lambda$, i.e., $f_b(\lambda) = \max \{ \alpha_i \lambda + \beta_i \mid i \in 1..k \},\ \forall \lambda \in \Lambda.$ The difficulty comes from the fact that $f_b(\lambda)$ is defined as a maximum of affine functions in $\lambda$. Consequently, the constraint $\c^t\x \leq f_b(\lambda)$ cannot be rewritten as a finite number of linear inequalities valid for all $\lambda$. Therefore, it is not possible to perform a global analysis over the entire domain. Instead, it is necessary to reason piece by piece, i.e., interval by interval, where $f_b(\lambda)$ is affine. More precisely, there is $p$ pieces of $f_b(\lambda)$ in intervals $(I_j)_p$ such that, on each of them, $f_b(\lambda)=\alpha_j\lambda+\beta_j$. By fixing an interval $I_j$ and taking $\lambda\in I_j$, the problem of defining $g^\varepsilon_b(\lambda)$ can then be written in the following form:
\begin{equation}\label{eq:gb_piece_en}
    \begin{aligned}
        g_b^{(j)}(\lambda)=\min_{\x}\quad & \boldsymbol{\varphi}^t\x\\
        \text{s.t.}\quad
        & A\x \leq \b + \lambda \d\\
        & \c^t\x \leq \alpha_j\lambda+\beta_j\\
        & \x \ge 0.
    \end{aligned}
\end{equation}

Then, $g_b(\lambda)$ can be studied by assembling the results obtained on each of the intervals. 
As with $h_c(\lambda)$, this piecewise analysis allows us to recover the standard behavior of parametric linear programs locally. On a fixed interval $I_j$, where the quasi-optimality constraint becomes affine in $\lambda$, the problem reduces to a standard right-hand side perturbation problem, and the usual sensitivity arguments apply. In particular, $g_b^{(j)}(\lambda)$ is a piecewise convex function that is linear on $I_j$ (it is linear as long as the optimal basis remains unchanged, and changes in slope can only occur when the optimal basis changes). Therefore, $g_b(\lambda)$ is a piecewise function composed of piecewise linear convex functions, obtained by concatenating the functions $g_b^{(j)}(\lambda)$ on the partition $(I_j)_p$. \\ Thus, as with $h_c(\lambda)$, this decomposition into subproblems is conceptually simple, and each subproblem can be handled using standard sensitivity tools. However, as Theorem \ref{th:murty} shows, the total number of pieces can be exponential in the worst case, which can make the calculation significantly more difficult.

\subsection{Dual form}\label{sub:dual_b}
\noindent
With regard to dual variables, here is the model : 
\begin{equation}
    \begin{aligned}
    h^\varepsilon_b(\lambda) = \min_{\y}\quad & \boldsymbol{\varphi}^t\y\\
    \text{s.t.}\quad & (\b+\lambda\d)^t \y \geq f_b(\lambda)-\varepsilon \\
    & A^t \y \le \c\\
    & \y \le 0.
    \end{aligned}
\end{equation}
Using the same principle as for $g^\varepsilon_c(\lambda)$, the model can be modified as follows: 
\begin{equation}
    \begin{aligned}
        h^\varepsilon_b(\lambda) = \min_\y& \quad \boldsymbol{\varphi}^t\y\\
        \text{s.t. } & (\b+\lambda\d)^t \y \ge \alpha_i \lambda + \beta_i-\varepsilon, \quad \forall i \in 1...k  \\ & A^t\y \le \c \\
         &  \y \le 0
    \end{aligned}
\end{equation}

Similar to $g^\varepsilon_c(\lambda)$, the function $h^\varepsilon_b(\lambda)$ can be analyzed using the same arguments, and the two models share the same properties. In particular, it has been shown in Theorem \ref{Berge_maximum} and Observation \ref{obs:equivalentA} that, for $\varepsilon>0$, the selected value function is continuous and can be reformulated as the optimal value of a parametric problem of the form $\P_A(\lambda)$. The same reasoning, therefore, applies to $h^\varepsilon_b(\lambda)$. Furthermore, on any interval where the optimal basis remains unchanged, the corresponding optimal solution (and therefore $h^\varepsilon_b(\lambda)$) is a homographic function. However, as shown by Theorem \ref{th:murty}, the total number of local pieces may be exponential in the worst case, which can make piecewise analysis computationally too expensive. This also justifies considering global approximation methods. In particular, since $h^\varepsilon_b(\lambda)$ can be viewed as a problem of type $\P_A(\lambda)$, the bounding methods developed by  \citet{MiftariEtAl2024LHSBounding} can also be applied here to compute lower and upper bounds over intervals of $\lambda$.

\end{document}